\documentclass[11pt,a4paper]{article}

\usepackage[utf8]{inputenc}
\usepackage{amsmath}
\usepackage{mathtools}
\usepackage{amsfonts}
\usepackage{bbm} %blackboard 1
\usepackage{bm} 
\usepackage{amssymb}
\usepackage{amsthm}
\usepackage{geometry}
\usepackage{mathtools}
\usepackage[noadjust]{cite}
\usepackage[usenames,dvipsnames]{xcolor}
\usepackage[vcentermath]{genyoungtabtikz}
\usepackage{tikz}
\usetikzlibrary{decorations.pathreplacing}
\usepackage{kpfonts}
\usepackage{subcaption}
\usepackage{graphicx}
\usepackage{setspace}
\usepackage[colorlinks, linkcolor=RoyalBlue, citecolor=RoyalBlue, urlcolor=RoyalBlue]{hyperref}
\usepackage{xcolor}
\usepackage{array}
\usepackage{multirow}
\newcolumntype{?}{!{\vrule width 1pt}}
\usepackage[final]{showlabels}
\theoremstyle{plain}
\newtheorem{proposition}{Proposition}[section]
\newtheorem{theorem}[proposition]{Theorem}
\newtheorem{corollary}[proposition]{Corollary}
\newtheorem{lemma}[proposition]{Lemma}

\theoremstyle{definition}
\newtheorem{definition}[proposition]{Definition}
\newtheorem{example}[proposition]{Example}
\newtheorem{remark}[proposition]{Remark}
\newtheorem*{remark*}{Remark}

\newcommand{\Sn}{\mathfrak{S}_n}
\newcommand{\An}{\mathfrak{A}_n}
\newcommand{\C}{\mathbb{C}}
\newcommand{\DD}{\mathbf{D}}
\newcommand{\Reg}[2]{\textup{Reg}_{#1}(#2)}
\newcommand{\F}{\mathbb{F}}

\newcommand{\Par}[1]{\textup{Par}(#1)}
\newcommand{\m}{\bm{m}}
\newcommand{\Bgot}{\mathfrak{B}}
\newcommand{\lessdom}{\trianglelefteq}
\newcommand{\ldom}{\trianglelefteq}
\newcommand{\domleq}{\unlhd}

\newcommand{\dr}{\partial}
\newcommand{\dreg}{\dr^{\text{reg}}}
\newcommand{\li}[1]{{}_{#1}}
\newcommand{\lessup}{\mathbin{\rotatebox[origin=c]{45}{$\trianglelefteq$}}}
\newcommand{\lesdow}{\mathbin{\rotatebox[origin=c]{-45}{$\trianglelefteq$}}}

\newcommand{\fn}[1]{\lceil #1 \rfloor}
\newcommand{\FN}[1]{\left\lceil #1 \right\rfloor}
\newcommand{\p}[1]{$p$-#1}

\newcommand\fg{\Yfillcolour{black!20}}

\newcommand\fw{\Yfillcolour{white}}

\newcommand{\lr}[1]{\langle #1 \rangle}
\newcommand{\Lr}[1]{\left\langle #1 \right\rangle}

\newcommand{\bd}{
\begin{tikzpicture}
\fill (0,0) circle (0.15);
\draw [line width=0.3mm](0,-0.4)--(0,0.4);
\end{tikzpicture}}
\newcommand{\pa}{
\begin{tikzpicture}
\draw [line width=0.3mm](-0.1,0)--(0.1,0);

\draw [line width=0.3mm](0,-0.4)--(0,0.4);
\end{tikzpicture}}

\newcommand{\hil}[1]{%
  \colorbox{black!20}{$\displaystyle#1$}}

\newcommand\xqed[1]{%
  \leavevmode\unskip\penalty9999 \hbox{}\nobreak\hfill
  \quad\hbox{#1}}
\newcommand\demo{\xqed{$\triangle$}}

\author{Ana Bernal \\ \footnotesize{Laboratoire de Math\'ematiques de Reims UMR9008 CNRS et Université de Reims Champagne-Ardenne,}\\ \footnotesize{Moulin de la Housse BP 1039, 51687 REIMS cedex 2, France.} \\ \href{mailto:ana.bernal@univ-reims.fr}{\footnotesize{\texttt{ana.bernal@univ-reims.fr}}}}

\title{Unitriangular basic sets for blocks of the symmetric and alternating groups of small weights}
\date{}

\begin{document}
\maketitle

%\tableofcontents

\begin{abstract}

We study the existence of unitriangular basic sets for the symmetric group which behave nicely with respect to the Mullineux involution. Such sets give a natural labelling for the modular irreducible representations. We show that, for any odd prime $p$, the $p$-blocks of the symmetric group with weight $2$ have stable unitriangular basic sets which we describe by studying the combinatorics of partitions in these blocks.

\end{abstract}

\section{Introduction}\label{intro}

Let $n$ be a positive integer and $\Sn$ the symmetric group on $n$ symbols. Let $p$ be an odd prime and let $\F_p$ be a field with $p$ elements. The simple $\C\Sn$-modules can be naturally indexed by the set of partitions of $n$:
\[
\textup{Irr}(\C\Sn)=\{S^\lambda \mid \lambda\  \text{is a partition of $n$}\}.
\]
The modules $S^\lambda$ are the so called Specht modules. Specht modules $S^\lambda$ are actually defined over $\mathbb{Z}$ and thus it is possible to ``reduce mod $p$'' and see $S^\lambda$ as an $\F_p\Sn$-module. In general, Specht modules over $\F_p$ are not simple, or not even completely reducible, but one can look at the composition multiplicities of the simple $\F_p\Sn$-modules in $S^\lambda$. These multiplicities give interesting information to study the representation theory of $\Sn$ over $\F_p$. For a simple $\F_p\Sn$-module $D$, let $d_{\lambda D}=[S^\lambda : D]$ be the composition multiplicity of $D$ in the Specht module $S^\lambda$. The numbers $d_{\lambda D}$ form then a matrix $\mathbf{D}_{n,p}$ with non-negative integer coefficients, rows indexed by the partitions of $n$ and columns indexed by the simple $\F_p\Sn$-modules. This is the \emph{decomposition matrix} of $\F_p\Sn$.

When the rows of $\mathbf{D}_{n,p}$ are organised decreasingly according to any total order on partitions which refines the dominance order, something of special interest happens: for each column, the first non-zero entry is $1$ and the associated row corresponds to a unique \emph{$p$-regular partition} of $n$. A $p$-regular partition is a partition in which no part is repeated $p$ or more times. This gives a natural indexing of the simple $\F_p\Sn$-modules by the set of $p$-regular partitions of $n$:
\begin{equation}\label{eq:simplefpsnmodules}
\textup{Irr}(\F_p\Sn)=\{D^\mu \mid \mu\  \text{is a $p$-regular partition of $n$}\}.
\end{equation}
For example, taking the lexicographic order on partitions, $\DD_{4,3}$ is
\[
\begin{array}{l?cccc}
4 & 1 \\
3,1 & \cdot & 1 \\
2^2 & 1 & \cdot & 1 \\
2, 1^2 & \cdot & \cdot & \cdot & 1 \\
1^4 & \cdot & \cdot & 1 & \cdot \\
\end{array}
\]
where the dots and ommited entries are equal to $0$. Indeed we can see that up to rearranging the rows and columns, the square submatrix formed by rows corresponding to $p$-regular partitions, has a lower unitriangular shape. The fact that the simple $\F_p\Sn$-modules can be indexed by the set of $p$-regular partitions is known before observing the form of the matrix $\mathbf{D}_{n,p}$. Indeed, there are as many classes of inequivalent simple $\F_p\Sn$-modules as conjugacy classes of $\Sn$ with order which is prime to $p$. Such conjugacy classes are, in turn, in bijection with the $p$-regular partitions of $n$. The additional information which comes from the matrix $\mathbf{D}_{n,p}$ having such a form, is that the classes of the Specht modules 
\begin{equation}\label{eq:fpsnspechtmodules}
\{S^\lambda \mid \lambda\  \text{is a $p$-regular partition of $n$}\},
\end{equation} 
form a $\mathbb{Z}$-basis of the so-called  \emph{Grothendieck group} of $\F_p\Sn$ (roughly speaking, the Grothendieck group of a finite dimensional algebra $A$ is the $\mathbb{Z}$-module generated by the isomorphism classes of the finitely generated $A$-modules, with the relation that says that two modules are equivalent iff they have the same composition factors). And the square sub-matrix of $\mathbf{D}_{n,p}$ with rows indexed by the $p$-regular partitions is precisely the transition matrix between this basis, formed by the classes of the Specht modules in (\ref{eq:fpsnspechtmodules}), and the natural basis of the Grothendieck group; namely the one formed by the classes of simple $\F_p\Sn$-modules (\ref{eq:simplefpsnmodules}). The set of $p$-regular partitions is then said to be a \emph{unitriangular basic set}, or \emph{UBS} for short, for $\F_p\Sn$. A UBS is a subset of the Specht modules, together with a total order of the rows, for which the matrix $\mathbf{D}_{n,p}$ takes such a form i.e. such a set indexes simple $\F_p\Sn$-modules, and to such a set we associate a $\mathbb{Z}$-basis of the Grothendieck group of $\F_p\Sn$. The general idea is that the Specht modules are much better known than the simple modules (dimensions, characters) and, roughly speaking, the decomposition matrix allows to pass from one the other.

Unitriangular basic sets for $\F_p\Sn$ are not unique. When studying the representation theory of the alternating group $\An$, the quest of finding a more ``suitable'' UBS arises. By classical Clifford theory, the behaviour of the restriction from $\Sn$ to $\An$ is determined by tensoring simple modules of $\Sn$ with the sign representation. The $\F_p\Sn$-module $D^\lambda \otimes \varepsilon$, where $\varepsilon$ is the sign representation of $\F_p\Sn$ and $\lambda$ is a $p$-regular partition, is simple and then it is isomorphic to $D^{\lambda^*}$, where $\lambda^*$ is some $p$-regular partition. In the characteristic zero setting ($\C$ instead of $\F_p$), or if $p>n$, the partition $\lambda^*$ is simply the \emph{conjugate} partition $\lambda'$ of $\lambda$. But for $p \leq n$, the description of $\lambda^*$ is much more difficult, $\lambda^*=\m(\lambda)$, where $\m$ is a map called the \emph{Mullineux map}. The Mullineux map is then an important involution on $p$-regular partitions corresponding with tensoring with the sign, and directly related to the representation theory of the alternating group. There are various combinatorial algorithms to find $\m(\lambda)$ (see \cite{mulli}, \cite{fordkleschev}, \cite{xu}, \cite{brundankujawa}, \cite{fayers2}), which are quite elaborate. Hence, one way in which a UBS for $\F_p\Sn$ could be more suitable is one for which tensoring with the sign corresponds with conjugation of partitions, even for $0<p<n$. Such a UBS is, in particular, stable by conjugation of partitions. In order to describe the aim of this paper let us first state more precise definitions.

\begin{definition}\label{def:UBSgroup} A \emph{unitriangular basic set} (UBS) for $\mathbb{F}_p\Sn$ is the datum of a triplet $(U,\leq,\Psi)$, where $U$ is a subset of partitions of $n$, $\leq$ is a total order defined on the set of partitions of $n$ and $\Psi$ is a bijection
\[
\Psi: U \longrightarrow \textup{Irr}(\mathbb{F}_p\Sn)
\]
satisfying:
\begin{itemize}
\item[(1)] for all $\mu \in U$, we have $d_{\mu\Psi(\mu)}=1$, and
\item[(2)] for all $D \in \textup{Irr}(\mathbb{F}_p\Sn)$ and $\lambda$ a partition of $n$, if $d_{\lambda D} \neq 0$ then $\lambda \leq \Psi^{-1}(D).$
\end{itemize}
The UBS $(U,\leq,\Psi)$ is said to be a \emph{stable unitriangular basic set} (SUBS) if, in addition
\begin{itemize}
\item[(A)] if $\mu \in U$, then $\mu' \in U$, and
\item[(B)] if $\mu=\mu'\in U$, then none of the diagonal hook-lengths of $\mu$ is divisible by $p$.
\end{itemize}
\end{definition}
\medskip

\noindent The condition (B) is a technical condition which allows the SUBS $(U,\leq,\Psi)$ to ``restrict'' to a UBS for $\mathbb{F}_p\An$, see \cite[Theorem 12]{brunatgramainjacon}. This condition also implies that the number of fixed points of conjugation in the set is equal to the number of fixed points of the Mullineux map.  Notice that the UBS of $\mathbb{F}_p\Sn$ formed by $p$-regular partitions is not a SUBS, since for example, the conjugate $\lambda'=(1^n)$ of the partition $\lambda=(n)$ is $p$-singular when $p \leq n$. The question is then: does a SUBS always exist for the symmetric group? And the answer is no. In \cite[\S 3]{brunatgramainjacon}, Brunat, Gramain and Jacon show that $\mathbb{F}_p\An$ does not always have a UBS (for example when $p=3$ and $n=18$), implying that the symmetric group cannot always have a SUBS. 

Despite this, the same question can be adressed ``blockwise''. From the modular representation theory of finite groups we know that the Specht modules and the simple modules of $\mathbb{F}_p\Sn$ fall into \emph{$p$-blocks}, which we describe precisely in \S \ref{sec:SUBSblocks}. This means in particular that the decomposition matrix is a block diagonal matrix:

\[
\mathbf{D}_{n,p} = 
\begin{pmatrix}
\begin{tikzpicture}
\draw (0,6) rectangle (0.5,6.6) ;
\draw (0.25,6.3) node {$*$};
\draw (0.5,5.4) rectangle (1,6) ;
\draw (0.75,5.7) node {$*$};
\draw (1,5.4)--(1.1,5.4);
\draw (1,5.4)--(1,5.3);
%\draw (1,4.8) rectangle (1.5,5.4) ;
\draw (1.25,5.2) node {$\ddots$};
\draw (1.5,4.8) -- (1.5,4.9);
\draw (1.5,4.8) -- (1.4,4.8);
\draw (1.5,4.2) rectangle (2,4.8) ;
\draw (1.75,4.5) node {$*$};
\draw (0.6,4.6) node {$0$};
\draw (1.5,6.1) node {$0$};
\end{tikzpicture}
\end{pmatrix}
\]
The concepts of UBS and SUBS can be adapted to have a meaning for each of these $p$-blocks, which we define precisely in \S \ref{sec:SUBSblocks}. The interest of doing this is to ask for the same properties blockwise: the submatrix corresponding to the block will have a lower unitriangular shape, this will give a natural labelling of the simple $\mathbb{F}_p\Sn$-modules in a block, for which the Mullineux map is conjugation. Another interest is to obtain a UBS for blocks of $\mathbb{F}_p\An$: if a block of $\mathbb{F}_p\Sn$ has a SUBS, then it provides by restriction, a UBS for blocks of $\mathbb{F}_p\An$. So that the question is to determine whether a block has a SUBS and if yes, to describe one.

Blocks of small weights of the symmetric group, and their decomposition matrices have been studied in detail, see for example \cite{scopes,fayers3,fayers4}. In this note, we give an explicit SUBS for any $p$-block of $\mathbb{F}_p\Sn$ of weight $2$. We use a labelling of the Specht modules in such blocks which comes from \cite{fayers} and \cite{richards}. In \cite{richards}, Richards introduces an object called a \emph{pyramid}, associated to a block. By means of the pyramid and the mentionned labels on the Specht modules, we obtain some properties on the dominance order in the blocks, which allow us to obtain a SUBS. Blocks of odd weight and pairs of conjugate blocks (of any weight) have explicitely known SUBS (\cite[\S 5.2]{brunatgramainjacon}), we explain this in detail in \S \ref{sec:oddweig}. The question for blocks of weight $0$ is automatically answered by the fact that such a block forms a SUBS by itself.

The organisation is as follows. In \S \ref{sec:2} we recall several definitions about partitions and we define (stable) unitriangular basic sets, or (S)UBS, for blocks of $\Sn$. In \S \ref{sec:4} we study self-conjugate blocks of weight $2$, that is, we study in detail combinatorial properties of partitions labelling Specht modules in these blocks which come mostly from \cite{richards}. We show what these properties imply in the dominance order on partitions in these blocks. The third part of this section contains the main result: we show how the properties stated in the first part allow to construct a SUBS for any self-conjugate $p$-block of weight $2$.

\section{Preliminaries}\label{sec:2}
	\subsection{Partitions, hooks and \texorpdfstring{$p$}{p}-cores}\label{sec:partitions}
		
In this section we recall some combinatorial objects linked with the representations of the symmetric group $\Sn$ and we set some notations. We take most of these notations from \cite{jameskerber}.

A \emph{partition} $\lambda$ is a weakly decreasing sequence $\lambda = (\lambda_1,\lambda_2,\ldots)$ of non-negative integers containing only finitely many non-zero terms. Let $n \in \mathbb{N}$ be such that $\lambda_1 + \lambda_2 + \cdots = n$. We say that $\lambda$ is a partition of $n$, which we write $\lambda \vdash n$ and $n$ is the \emph{rank} of $\lambda$. For $n \in \mathbb{N}$,  we denote by $\textup{Par}(n)$ the set of partitions of $n$. The integers $\lambda_i$ are called the \emph{parts} of the partition $\lambda$. If there is a part that repeats $k$ times, say $\lambda_i=\lambda_{i+1}=\cdots=\lambda_{i+k-1}$, we denote $\lambda$ as $(\ldots,\lambda_i^k,\ldots)$. The number of non-zero parts is the \emph{length} of $\lambda$ and is denoted $l(\lambda)$. The \emph{empty partition} is the only partition of $0$, it is denoted $\emptyset$.

The \emph{Young diagram} of a partition $\lambda$ is the set \[
[\lambda]=\{(i,j) \in \mathbb{N} \times \mathbb{N} \mid i\geq 1  \ \ \text{and} \ \ 1 \leq j \leq \lambda_i \},
\] 
whose elements are called \emph{nodes}. We represent $[\lambda]$ as an array of boxes in the plane with the convention that $i$ increases downwards and $j$ increases from left to right. A partition is often identified with its Young diagram. The Young diagram of $\lambda=(4,2^2,1)$ is 
\[
[\lambda]=\yngs(1,4,2,2,1)
\]

The \emph{conjugate} (or \emph{transpose}) of a partition  $\lambda=(\lambda_1,\ldots,\lambda_l)$ of $n$, where $l=l(\lambda)$, is the partition of $n$ denoted $\lambda'$ and defined by $\lambda'_i=\#\{j \mid 1\leq j \leq l \ \text{and}\ \lambda_j \geq i\}$. If $\lambda=(4,2^2,1)$, as above, then $\lambda'=(4,3,1^2)$. Conjugation of partitions is easily seen in the Young diagram; the Young diagram of $\lambda'$ is the reflection of the Young diagram of $\lambda$ against the main diagonal.
\[
[\lambda']=\yngs(1,4,3,1^2)
\]
A partition which is equal to its conjugate is called a \emph{self-conjugate} partition. Its Young diagram is then symmetrical with respect to the main diagonal.

For a positive integer $p$, the partition $\lambda$ is said to be \emph{$p$-regular} if it does not contain $p$ parts $\lambda_i \neq 0$ which are equal. A partition which is not $p$-regular is called \emph{$p$}-singular. The partition $\lambda=(4,2^2,1)$ above is not $2$-regular but it is $3$-regular. We denote by $\Reg{p}{n}$ the set of $p$-regular partitions of $n$.

The \emph{dominance order} ``$\trianglelefteq$'' is a partial order defined on the set of all partitions. It is defined as follows: let $\lambda$ and $\mu$ be two partitions, we say that $\lambda \trianglelefteq \mu$ if and only if
\[
\sum_{i=1}^k \lambda_i \leq \sum_{i=1}^k \mu_i \quad \text{for every }k\geq1.
\]

We write $\lambda \triangleleft \mu$ if $\lambda \lessdom \mu$ but $\lambda \neq \mu$. The \emph{lexicographic order} ``$\leq$'' is a total order defined on the set of all partitions. It is defined as follows: let $\lambda$ and $\mu$ be two partitions, we say that $\lambda \leq  \mu$ if and only if the first non-vanishing difference $\mu_i-\lambda_i$ is positive. We write $\lambda < \mu$ if $\lambda \leq \mu$ but $\lambda \neq \mu$. The lexicographic order is a refinement of the dominance order in the set $\Par n$, and they are different if $n \geq 6$.

We now recall some particular sets of nodes in the Young diagram of a partition $\lambda$. To any node $(i,j)$ of the diagram of $\lambda$ we can associate its \emph{$(i,j)$-th hook}, denoted $H^\lambda_{i,j}$, which consits of: the node $(i,j)$, the nodes to the right of it in the same row (the \emph{arm}), and the nodes lower down in the same column (the \emph{leg}). The size of $H^\lambda_{i,j}$ is the \emph{length} (or \emph{hook-length}) of the $(i,j)$-th hook, it is equal to $\lambda_i-j+\lambda_j'-i+1$. The number of nodes on the arm (resp. leg) of the hook is called the \emph{arm-length} (resp. \emph{leg-length}). The \emph{rim} of $\lambda$ is the set of nodes $\{ (i,j) \in [\lambda] \mid (i+1,j+1) \notin [\lambda]\}$, in words, it is the south-east border of $[\lambda]$. To $H^\lambda_{i,j}$ (or to $(i,j)$) we can associate a set $R^\lambda_{i,j}$ of the same size, consisting of (adjacent) nodes in the rim of $\lambda$: the upper right node of $R^\lambda_{i,j}$ is the node $(i,\lambda_i)$ and the lower left node is $(\lambda_j',j)$. The set $R^\lambda_{i,j}$ is called the $(i,j)$-th \emph{rim hook}. The hook $H^{(5,4,3,1)}_{1,2}$ is illustrated by shaded nodes and the corresponding rim-hook $R^{(5,4,3,1)}_{1,2}$ illustrated by nodes marked with ``$\times$'' in the following diagram
\[
\gyoung(;!\fg;;;\times;\times!\fw,;!\fg;!\fw;\times;\times,!\fw;!\fg;\times!\fw;\times,;)
\]
Both the hook and the rim-hook in this diagram are of length $6$, the arm-length is $3$ and the leg-length is $2$. We call \emph{$p$-rim-hook} any rim-hook of length $p$, for a positive integer $p$. 
 
Removing a rim-hook from a partition, that is, removing the nodes of a rim-hook from its Young diagram, results again in the Young diagram of a partition. Let $p$ be a positive integer, and $\lambda \vdash n$ a partition. Consider the following procedure : if $\lambda$ has a $p$-rim-hook, remove it from $\lambda$. We obtain a partition of $n-p$. If the partition obtained has a $p$-rim-hook, remove it, and we get a partition of $n-2p$. Continue removing $p$-rim-hooks in this way until obtaining a partition $\gamma_\lambda$, possibly the empty partition, for which the diagram does not have $p$-rim-hooks. Suppose that in total we removed $w$ $p$-rim-hooks in this sequence of steps. The partition obtained $\gamma_\lambda \vdash (n-wp)$ does not depend on the way that the $p$-rim-hooks were removed (\cite[Theorem 2.7.16]{jameskerber}). This partition is called the \emph{$p$-core} of $\lambda$. The number $w$ is also independent of the way they were removed and it is called the \emph{$p$-weight} of $\lambda$. The $p$-weight of $\lambda$ is also equal to the number of rim-hooks of $\lambda$ of length divisible by $p$ (\cite[2.7.40]{jameskerber}). A partition without $p$-rim-hooks is called a \emph{$p$-core}, in other words, a partition of $p$-weight $0$.

A \emph{$p$-BG-partition}, or simply a \emph{BG-partition} is a self-conjugate partition $\lambda$ such that none of the positive hook-lengths $|H^\lambda_{i,i}|$ is divisible by $p$.

We recall the abacus diagram for a partition, which was introduced by G. James in \cite{james2}.  For a positive integer $p$, a $p$-abacus is a diagram with $p$ vertical lines called \emph{runners} on which there are positions labelled by the integers as follows. If $p=3$, the positions are 

 \begin{center}
 \begin{tikzpicture}
 \draw [dotted] (0,3-.3)--(0,2.3);
 \draw [dotted] (1,3-.3)--(1,2.3);
 \draw [dotted] (2,3-.3)--(2,2.3);
 \draw (2,2) node{$-4$};
 \draw (1,2) node{$\cdots$};
 \draw (0,2) node{$\cdots$};
 \draw (0,2-.3)--(0,1.3);
 \draw (1,2-.3)--(1,1.3);
 \draw (2,2-.3)--(2,1.3);

 \draw (0,1-.3)--(0,0.3);
 \draw (1,1-.3)--(1,0.3);
 \draw (2,1-.3)--(2,0.3);
 \draw (0,1) node{$-3$};
 \draw (1,1) node{$-2$};
 \draw (2,1) node{$-1$};
  \draw (0,0-.3)--(0,-0.7);
 \draw (1,0-.3)--(1,-0.7);
 \draw (2,0-.3)--(2,-0.7);
 
  \draw (0,-1-.3)--(0,-1.7);
 \draw (1,-1-.3)--(1,-1.7);
 \draw (2,-1-.3)--(2,-1.7);
 \draw (0,0) node{$0$};
 \draw (1,0) node{$1$};
 \draw (2,0) node{$2$};
   \draw [dotted](0,-2-.3)--(0,-2.7);
 \draw [dotted](1,-2-.3)--(1,-2.7);
 \draw [dotted](2,-2-.3)--(2,-2.7);
 
 \draw (0,-1) node{$3$};
 \draw (1,-1) node{$4$};
 \draw (2,-1) node{$5$};
 
 \draw (0,-2) node{$6$};
 \draw (1,-2) node{$\cdots$};
 \draw (2,-2) node{$\cdots$};
 \end{tikzpicture}
 \end{center}
 
These positions give a natural labelling for runners: The $0$th runner is the one with positions which are congruent to $0$ mod $p$, etc. The set of \emph{beta-numbers} of $\lambda$ is the set $\{\beta_i=\lambda_i-i \mid i \geq 1\}$ (it is an infinite set of integers). The \emph{$p$-abacus} for $\lambda$ (or simply \emph{abacus}) is the abacus with \emph{beads} placed in positions $\beta_i$, for $i \geq 1$, and empty positions elsewhere. For example, if $p=5$ and $\lambda=(4,3^3,2,1,1)$, the beta-numbers for $\lambda$ are $\{3, 1 , 0 ,-1 , -3 , -5 , -6 , -8 , -9 , -10, \ldots\}$ the $5$-abacus of $\lambda$ is

\[
\begin{matrix}
\vdots & \vdots & \vdots & \vdots & \vdots\\
\bd & \bd & \bd & \bd & \bd \\[-6pt]
\bd & \bd & \bd & \pa & \bd \\[-6pt]
\bd & \pa & \bd & \pa & \bd \\[-6pt]
\bd & \bd & \pa & \bd & \pa \\[-6pt]
\pa & \pa & \pa & \pa & \pa \\[-6pt]
\vdots & \vdots & \vdots & \vdots & \vdots\\
\end{matrix}
\]

Adding $1$ to each of the beta-numbers for $\lambda$ and placing beads in an abacus in these new positions produces another abacus configuration for $\lambda$. But a given abacus configuration corresponds to a unique partition. To remove this ambiguity, we always chose an abacus configuration with a number of beads which is a ``multiple of $p$'', meaning that the (finite) number of beads, starting to count the beads from any row full of beads, with all rows above also full of beads, and all the way down on every runner, is a multiple of $p$. This way of drawing the abacus guarantees that there is a unique $p$-abacus for a partition.  

If we add $p$ nodes to the Young diagram of $\lambda$, in such a way that they form a $p$-rim-hook, the $p$-abacus of the obtained partition is obtained from the abacus of $\lambda$ by just sliding a bead one space down \cite[Lemma 2.7.13]{jameskerber}. So, adding a $p$-rim-hook is sliding a bead one space down, and conversely removing a $p$-rim-hook is sliding a bead one space up. Hence, the abacus of the $p$-core of $\lambda$ is obtained from the abacus of $\lambda$ by sliding the beads on every runner all the way up. And, conversely, any partition $\lambda$ of $p$-weight $w$ is obtained from the abacus of its $p$-core by making $w$ sliding down individual movements of beads. It is easy to see that the leg-length of the rim-hook added by sliding a bead one space down is equal to the number of beads between the initial and final positions of the moved bead.

\paragraph{Labelling of runners in the abacus} We recall a useful labelling for the runners of the abacus of a $p$-core. Let $\gamma$ be a $p$-core and consider its abacus. Take the positions of the lowest bead in each runner. This gives a list of $p$ integers $\rho_0<\rho_1<\cdots<\rho_{p-1}$. Since each of these integers corresponds to (a position on) a runner, such a list yields a labelling $0,1,\ldots,p-1$ of the runners, possibly different from the natural labelling: now the $0$-th runner is the leftmost with ``\emph{minimal number}'' of beads, the $(p-1)$-th runner is the rightmost with the ``\emph{maximal number}'' of beads. 
We can rephrase this by defining a total order on runners: Let ``$\lessdot$'' be the following order on runners. For $R$ and $S$ two runners on the abacus of $\gamma$ we say that $R \lessdot S$ iff $R$ has strictly less beads than $S$, or $R$ and $S$ have the same number of beads and $R$ is to the left of $S$. Here, by ``\emph{number of beads} in a runner'' we mean the number of beads below a certain fixed threshold, which is a row such that all rows above are full of beads.

Now, with this order, our new labelling $0,1,\ldots,p-1$ coincides with increasingly ordering the runners with respect to ``$\lessdot$''. From now on, if not otherwise specified, we use this labelling for runners on an abacus. 

\begin{example}\label{ex:w1new}
Let $p=5$ and $\gamma=(2,1)$. The Young diagram and $5$-abacus of $\gamma$, with the corresponding labelling of runners is  
\[
\begin{array}{c}
\gamma \\
\\
\gyoung(;;,;)
\end{array} 
\qquad \qquad
\begin{matrix}
1 & 4 & 2 & 0 & 3\\
\bd & \bd & \bd & \bd & \bd \\[-15pt]
\bd & \bd & \bd & \bd & \bd \\[-15pt]
\bd & \bd & \bd & \pa & \bd \\[-15pt]
\pa & \bd & \pa & \pa & \pa \\[-15pt]
\pa & \pa & \pa & \pa & \pa  \\[-15pt]
\end{matrix}
\]
\demo
\end{example}

	\subsection{Blocks and stable unitriangular basic sets for blocks}\label{sec:SUBSblocks}

This section contains definitions and notations regarding $p$-blocks of $\F_p\Sn$ and unitriangular basic sets for blocks. Before stating these definitions we make some remarks.

In \S \ref{intro} we recalled the definition of a unitriangular basic set for $\F_p\Sn$, we do not recall it here. However, we remark that in the particular case of the UBS $\Reg{p}{n}$, the total order involved is any total order refining the dominance order on partitions. Hence, condition (2) on Definition \ref{def:UBSgroup} is stronger: for all $D^\mu \in \textup{Irr}(\mathbb{F}_p\Sn)$ and $\lambda \vdash n$ we have the following:
\begin{equation}\label{eq:conditiondominance}
\text{If }\  d_{\lambda D^\mu} \neq 0\  \ \text{then}\ \  \lambda \trianglelefteq \mu.
\end{equation}
This fact will be used later.

There is a more general/weaker notion than that of a unitriangular basic set, that is a \emph{basic set}. A basic set for $\F_p\Sn$ is a subset of $\Par{n}$ in bijection with $\textup{Irr}(\F_p\Sn)$, such that the corresponding square submatrix of $\DD_{n,p}$ is invertible over $\mathbb{Z}$.

The notions of decomposition matrices, basic sets and UBSs are defined in the much more general setting of modular representation theory of finite groups, often in terms of Brauer characters. 

The group algebra $\F_p\Sn$ decomposes into a direct sum of indecomposable subalgebras, which are called the \emph{$p$-blocks} of $\F_p\Sn$ (or simply the \emph{blocks})
\[
\F_p\Sn = \Bgot_1 \oplus \Bgot_2 \oplus \cdots \oplus \Bgot_r
\]
We say that a $\F_p\Sn$-module $M$ \emph{belongs} to the block $\Bgot_i$ if $\Bgot_iM=M$ and $\Bgot_j$ annihilates $M$ for $j\neq i$ (see \cite{alperin} for the more general setting of blocks of a finite-dimensional algebra). From modular representation theory, we know that each Specht module belongs to a unique block of $\F_p\Sn$, as well as each simple module. By the Nakayama conjecture (\cite[6.1.2]{jameskerber}), the $p$-blocks of $\Sn$ are classified: two Specht modules $S^\lambda$ and $S^\mu$ belong to the same $p$-block if and only if $\lambda$ and $\mu$ have the same $p$-core. Hence, each block $\Bgot_i$ can be indexed by the corresponding $p$-core. We say that the block $\Bgot_\gamma$, where $\gamma$ is the corresponding $p$-core, is a block of $p$-weight $w$ if the partitions $\lambda$ such that $S^\lambda$ belongs to $B_\gamma$ have $p$-weight $w$. Throughout this note, by abuse of notation, $\Bgot_\gamma$ is the set of partitions $\lambda$ such that $S^\lambda$ is in $\Bgot_\gamma$.

We now define UBS and SUBS for a $p$-block.

\begin{definition}\label{def:unibl} Let $\gamma \vdash (n-wp)$ be a $p$-core. Let $\Bgot_\gamma$ be the corresponding block of $\mathbb{F}_p\Sn$. A \emph{unitriangular basic set} (\emph{UBS}) for $\Bgot_\gamma$ is the data of a triplet $(U_\gamma\ ,\ \leq\ ,\ \Psi_\gamma)$, where $U_\gamma \subseteq \Bgot_\gamma$ , $\leq$ is a total order defined on $\Bgot_\gamma$ and $\Psi$ is a bijection
\[
\Psi_\gamma: U_\gamma \longrightarrow \textup{Irr}_p(\Bgot_\gamma) \subseteq \textup{Irr}(\mathbb{F}_p\Sn)
\]
satisfying:
\begin{itemize} 
\item[(1)] for all $\mu \in U_\gamma$, we have $d_{\mu\Psi(\mu)}=1$, and
\item[(2)] for all $D \in \textup{Irr}_p(\Bgot_\gamma)$ and $\lambda \in \Bgot_\gamma$, if $d_{\lambda D}\neq 0$ then $\lambda \leq \Psi^{-1}(D).$
\end{itemize}
\end{definition}

Here, $\textup{Irr}_p(\Bgot_\gamma)$ denotes the set of simple $\F_p\Sn$-modules in the block $\Bgot_\gamma$ up to equivalence. It is clear that any UBS for $\mathbb{F}_p\Sn$ restricts to an UBS for each block. For example the set $\Reg{p}{\Bgot_\gamma}$ of $p$-regular partitions in $\Bgot_\gamma$, which is equal to $\Reg{p}{n} \cap \Bgot_\gamma$, is a UBS for $\Bgot_\gamma$.

Recall, from the introduction, that we look for a stable UBS for a block (see Definition \ref{def:UBSgroup}) of the symmetric group, let us state the precise definition of this. For a $p$-block $\Bgot_\gamma$, whose partitions have $p$-core equal to $\gamma$, we call its \emph{conjugate block}, the block $\Bgot_{\gamma'}$ consisting of the partitions of $n$ with $p$-core $\gamma'$. If $\gamma=\gamma'$, then $\Bgot_\gamma=\Bgot_{\gamma'}$ and we say that $\Bgot_\gamma$ is a \emph{self-conjugate block}.

\begin{example} The block of $\F_5\mathfrak{S}_8$ associated to the $5$-core $\tau=(3)$ is
\[
\Bgot_\tau=\{(8),\ (4^2),\ (3^2, 1^2),\ (3, 2, 1^3),\ (3, 1^5)\}.
\]
Its conjugate block is the block $\Bgot_{\tau'}$ associated to the $5$-core $\tau'=(1^3)$, whose partitions are all the conjugates of partitions in $\Bgot_\tau$:
\[
\Bgot_{\tau'}=\{(6, 1^2),\ (5, 2, 1),\ (4, 2^2),\ (2^4),\ (1^8)\}.
\]
The block of $\F_5\mathfrak{S}_8$ associated to the $5$-core $\gamma=(2,1)$ is a self-conjugate block since $\gamma=\gamma'$.
\[
\Bgot_{\gamma}=\{(7, 1),\ (5, 3),\ (3^2, 2),\ (2^3, 1^2),\ (2, 1^6)\}.
\]
\demo
\end{example}

\begin{definition}\label{def:SUBSblock} Let $\gamma \vdash (n-wp)$ be a $p$-core. Let $\Bgot_\gamma$ be the corresponding block of $\mathbb{F}_p\Sn$, and $\Bgot_{\gamma'}$ its conjugated block. And suppose that $U_\gamma$ and $U_{\gamma'}$ are UBSs for $\Bgot_{\gamma}$ and $\Bgot_{\gamma'}$, respectively. Let $U=U_{\gamma} \cup U_{\gamma'}$. The set $U$ is a \emph{stable unitriangular basic set} (\emph{SUBS}) for $(\Bgot_\gamma,\Bgot_{\gamma'})$ if the following holds:
\begin{itemize}
\item[(A)] If $\lambda \in U$, then $\lambda' \in U$, and
\item[(B)] if $\lambda=\lambda'\in U$, then $\lambda$ is a $p$-BG-partition. That is, none of the diagonal hook-lengths $|H^\lambda_{i,i}|$ of $\lambda$ is divisible by $p$.
\end{itemize}
If $\Bgot_\gamma$ is self-conjugate and (A) and (B) hold for $U=U_\gamma$, we say that $U=U_\gamma$ is a SUBS for $\Bgot_\gamma$.
\end{definition}

Having a SUBS for $(\Bgot_\gamma,\Bgot_{\gamma'})$ implies that the sub-matrix of $\DD_{n,p}$ formed by the rows indexed by the set $U$ is lower unitriangular. We denote by $\DD_\gamma$ the (block-)sub-matrix of $\DD_{n,p}$ corresponding to the rows indexed by $\Bgot_\gamma$.

\subsubsection{SUBS for blocks of odd weight or not self-conjugate}\label{sec:oddweig}

As said in the introduction, given that there is not always a SUBS for $\F_p\Sn$, it is interesting to have a SUBS for some blocks. By the following proposition, which has been adapted to our notation, this question is already solved for blocks of odd weight and blocks which are not self-conjugate. 

\begin{proposition}[\text{\cite[Theorem 38]{brunatgramainjacon}}]\label{prop:SUBSodd} Let $\gamma \vdash (n-wp)$ be a $p$-core. Let $\Bgot_\gamma$ be the corresponding block of $\mathbb{F}_p\Sn$. Consider the following subset of $\Bgot_\gamma \cup \Bgot_{\gamma'}$ 
\[
U=\{\lambda \in \Reg{p}{\Bgot_\gamma} \mid \m(\lambda)<\lambda\}\ \ \sqcup \  \{\lambda' \mid \lambda \in \Reg{p}{\Bgot_\gamma} \ \text{and} \  \m(\lambda)<\lambda\} \ \sqcup \ \{\lambda \in \Reg{p}{\Bgot_\gamma} \mid \m(\lambda)=\lambda\},
\]
where $\m$ denotes the Mullineux map. If $w$ is odd or if $\gamma$ is not self-conjugate, then there exist UBSs $(U_\gamma\ ,\ \preceq\ ,\ \Psi_\gamma)$ and $(U_{\gamma'}\ ,\ \preceq\ ,\ \Psi_{\gamma'})$, for $B_\gamma$ and $B_{\gamma'}$, respectively such that $U=U_\gamma \cup U_{\gamma'}$ is a SUBS for $(\Bgot_\gamma,\Bgot_{\gamma'})$.
\end{proposition}

We make some comments on this proposition. In \cite[Proposition 33]{brunatgramainjacon}, Brunat, Gramain and Jacon show that the set 
\begin{equation}\label{eq:UBSpn}
U_{p,n}=\{\lambda \in \Reg{p}{n} \mid \m(\lambda)<\lambda\}\ \ \sqcup \  \{\lambda' \mid \lambda \in \Reg{p}{n} \ \text{and} \  \m(\lambda)<\lambda\} \ \sqcup \ \{\lambda \in \Reg{p}{n} \mid \m(\lambda)=\lambda\},
\end{equation}
together with some particular total order on $\Par{n}$, and a bijection $\Psi: U_{p,n} \rightarrow \textup{Irr}(\F_p\Sn)$, is a UBS for $\F_p\Sn$. The bijection $\Psi$ is defined as follows. Write 
\[
U_{p,n}=U_{p,n}^1\  \sqcup\ U_{p,n}^2 \ \sqcup\  U_{p,n}^3\ ,
\] 
where $U_{p,n}^1$, $U_{p,n}^2$ and  $U_{p,n}^3$ correspond to the three subsets in (\ref{eq:UBSpn}) above. Then $\Psi$ is defined as
\[
\begin{array}{cccl}
\Psi: & U_{p,n} & \longrightarrow & \textup{Irr}(\F_p\Sn) \\
\\
& \lambda & \longmapsto & 
\begin{cases*}
D^\lambda & if $\lambda \in U_{p,n}^1\ \cup \ U_{p,n}^3$,\\
D^{\m(\lambda')} & if $\lambda \in U_{p,n}^2$.\\
\end{cases*}
\end{array}
\]

In general, $U_{p,n}$ is not a SUBS for $\F_p\Sn$, since it contains the self-Mullineux partitions which could fail to verify properties (A) or (B) in Definition \ref{def:UBSgroup}. However, since the restriction of $U_{p,n}$ to each block results into a UBS, we can still wonder for which blocks such restriction results in a SUBS for the block. Proposition \ref{prop:SUBSodd} answers to this by giving a sufficient condition.

The reason why the restriction of the UBS $U_{p,n}$ to blocks of odd weight or blocks which are not self-conjugate gives a SUBS is that, for such blocks, the restriction of $U_{p,n}^3$ to the block is empty: the block does not contain self-Mullineux partitions. Let us see why. 

First for blocks which are not self-conjugate: from \cite{mullineux2}, we know that the $p$-core of a partition $\lambda$ and that of its image $\m(\lambda)$ under the Mullineux map, are conjugates. This says in particular that self-Mullineux partitions only occur in self-conjugate blocks. Then, if $\Bgot_\gamma$ is not a self-conjugate block, the set $U_{p,n}\cap(\Bgot_\gamma\cup \Bgot_{\gamma'})$ is a SUBS, by restricting the total order on $\Par{n}$. Indeed, this restriction does not contain any self-Mullineux partitions (any partition from  $U_{p,n}^3$), or any self-conjugate partition, and then by construction of $U_{p,n}$, it satisfies conditions (A) and (B) in Definition \ref{def:SUBSblock}.

Now, the following lemma shows why self-conjugate blocks of odd weight do not have any self-Mullineux partitions. It also contains a known fact about the number of self-Mullineux partitions in a block of even weight. A combinatorial proof of this fact can be found in \cite[Theorem 3.5]{bessenrodtolsson2}.

\begin{lemma}\label{lem:evenweight} Let $\gamma \vdash (n-pw)$ be a self-conjugate $p$-core. Let $\Bgot_\gamma$ be the corresponding block of $\F_p\Sn$. Then, the set of self-Mullineux partitions in $\Bgot_\gamma$ is non-empty if and only if $w$ is even. In this case, there are as many self-Mullineux partitions as $(\frac{p-1}{2})$-multipartitions of rank $\frac{w}{2}$.
\end{lemma} 

\begin{proof}
Denote by $\mathcal{M}_p^n(\gamma)$ and by $\text{BG}_p^n(\gamma)$ the set of self-Mullineux partitions and $p$-BG-partitions in $\Bgot_\gamma$, respectively. By \cite[Proposition 6.1]{brunatgramain} we know that, $|\mathcal{M}_p^n(\gamma)|=|\text{BG}_p^n(\gamma)|$.

In a fixed block, any partition is completely determined by its \emph{$p$-quotient}. For a definition of $p$-quotient see \cite[2.7.30]{jameskerber}. The $p$-quotient of a partition $\nu$ is a $p$-tuple of partitions $(\nu^{(1)},\nu^{(2)},\ldots,\nu^{(p)})$, or a \emph{$p$-multipartition} for which the sum of their ranks is $w$, the $p$-weight of $\nu$. In this case $w$ is the rank of this multipartition.

From \cite{brunatgramain}, choosing an appropriate convention, the $p$-quotient $(\nu^{(1)},\nu^{(2)},\ldots,\nu^{(p)})$ of $\nu \in \text{BG}_p^n(\gamma)$ is of the form
\[
(\nu^{(1)}, \nu^{(2)},\ldots, \nu^{(\frac{p-1}{2})}, \emptyset , \nu^{(\frac{p-1}{2})}{}',\ldots, \nu^{(2)}{}',\nu^{(1)}{}' ),
\] 
where $\nu^{(i)}{}'$ is the conjugate partition of $\nu^{(i)}$. Hence the rank, $w$, of this $p$-quotient is even: it is twice the sum of the ranks of $\nu^{(1)}, \nu^{(2)},\ldots, \nu^{(\frac{p-1}{2})}$. Also, this $p$-quotient is completely determined by these $\frac{p-1}{2}$ partitions. Thus, $|\mathcal{M}_p^n(\gamma)|$ is equal to the number of $(\frac{p-1}{2})$-multipartitions of $\frac{w}{2}$.

Conversely, if $w$ is even, each $(\frac{p-1}{2})$-multipartition of $\frac{w}{2}$ determines a unique $p$-quotient 
\[
(\nu^{(1)}, \nu^{(2)},\ldots, \nu^{(\frac{p-1}{2})}, \emptyset , \nu^{(\frac{p-1}{2})}{}',\ldots, \nu^{(2)}{}',\nu^{(1)}{}' ),
\] 
which corresponds to a $p$-BG-partition in the block $\Bgot_\gamma$.
\end{proof}

\section{SUBS for blocks of weight \texorpdfstring{$2$}{2}}\label{sec:4}

Throughout this section $\gamma \vdash (n-2p)$ is a self-conjugate $p$-core (unless otherwise specified) and $\Bgot_\gamma$ is the corresponding block of $\F_p\Sn$ of weight $2$. 

\subsection{Some combinatorics and notations for partitions in blocks of weight \texorpdfstring{$2$}{2}}

In \cite{richards}, Richards studied blocks of weight $2$, the combinatorics of partitions in such blocks and he gave a complete description of the decomposition numbers. In this section we recall some of his definitions and results which we use. Within his definitions there is an object associated to a block called pyramid, which will be the main object in this section. We also use notations from \cite{fayers}, in which Fayers presents an efficient way to label simple modules.

Consider the abacus display for $\gamma$. Recall, from the last paragraph in \S \ref{sec:partitions}, that positions of the lower beads are $\rho_0, \rho_1, \ldots, \rho_{p-1}$.   The \emph{pyramid} of $\gamma$ is a triangular array $(\li i\gamma_{j})$ of $0$s and $1$s, defined as follows: for $0\leq i\leq j \leq p-1$ let
 
 \[
 \li{i}\gamma_j = 
 \begin{cases*}
 1 & if $\rho_j-\rho_i<p$, \\
 0 & if $\rho_j-\rho_i>p$.
 \end{cases*}
 \]
 We organise these numbers in a diagram as follows
\begin{equation}\label{eq:pyramidgeneral}
 \begin{array}{ccccccccccccc}
 \text{\underline{row}} &\phantom{00000}\\
 p-1&\qquad&&&&&&&&\li 0 \gamma_{p-1} \\
 \\
 \vdots&&&&&&&&&\vdots \\
 \\
 \\
  2&&&&& \li 0 \gamma_2 && \li 1 \gamma_3 &&\cdots& \li{p-3} \gamma_{p-1} \\
 1&&&& \li 0 \gamma_1 & & \li 1 \gamma_2 &&\li 2 \gamma_3&\cdots&&\li{p-2} \gamma_{p-1}&\\
  0&&&\li 0 \gamma_0 & & \li 1 \gamma_1 &&\li 2 \gamma_2 && \cdots&& & \li{p-1} \gamma_{p-1}
 \end{array}
 \end{equation}
 
 For short, we write $\li i 0_j$ if $\li i \gamma_j=0$ and $\li i 1_j$ if $\li i \gamma_j=1$. The definition of the pyramid can be extended for convenience by allowing $i$ and $j$ to be any integers: if $i>j$, then $\li i \gamma_j=1.$ Otherwise, if $i<0$ or $j \geq p$ we define $\li i \gamma_j=0.$ In the pyramid diagram, that means that the outside upper left and upper right is filled with $0$s and the outside lower part is filled with $1$s. 
 For $0\leq k \leq p-1$, we call the \emph{$k$-th row} of the pyramid is the set of entries $\li i\gamma_j$ such that $j-i=k$. The \emph{apex} of the pyramid is the entry $\li 0 \gamma_{p-1}$.

 \begin{example}\label{ex:5core} Let $p=5$ and the $5$-core $\gamma=(2,2)$. The $5$-abacus for $\gamma$ is 
 \[
\begin{matrix}
3 & 4 & 2 & 0 & 1\\
\vdots & \vdots & \vdots & \vdots & \vdots\\
\bd & \bd & \bd & \bd & \bd \\[-6pt]
\bd & \bd & \bd & \pa & \pa \\[-6pt]
\bd & \bd & \pa & \pa & \pa \\[-6pt]
\pa & \pa & \pa & \pa & \pa \\[-6pt]
\vdots & \vdots & \vdots & \vdots & \vdots\\
\end{matrix}
\]
The positions of the lower beads in each runner are $11, 10, 7, 3,$ and $4$. Then, organised increasingly, $(p_0,p_1,p_2,p_3,p_4)=(3,4,7,10,11)$, so that the labelling of the runners is $3,4,2,0,1$ from left to right. The pyramid for $\gamma$ is 
\[
\begin{array}{ccccccccc} 
&&&&\li 0 0_4\\
&&&\li 0 0_3 && \li 1 0_4 \\
&& \li 0 1_2 && \li 1 0_3 && \li 2 1_4 \\
& \li 0 1_1&  & \li 1 1_2 &  & \li 2 1_3 &&\li 3 1_4 &\\
\li 0 1_0 & & \li 1 1_1 & & \li 2 1_2 & &\li 3 1_3 & &\li 4 1_4 
\end{array}
\]

\demo
 \end{example}
 
 The definition of the pyramid implies that if an entry is $1$, then the two entries just below are $1$ as well. Hence, below a $1$ there is a whole triangle or pyramid of $1$s. Hence if an entry is $0$, there are only $0$s above it.
 
 \begin{remark*}
 Pyramids are originally defined for $p$-blocks ($p$-cores) of any weight. Two different pyramids correspond to different $p$-cores. But two different $p$-cores can have the same pyramid. Two blocks have the same pyramid only if they are \emph{Scopes}-equivalent: there is a bijection between partitions in two Scopes-equivalent blocks such that the decomposition matrices in such blocks are equal. So that pyramids index Scopes-equivalence classes of blocks. See \cite[\S 3]{richards}.
 \end{remark*}
 
 \subsubsection{Notation for partitions in blocks of weight 2}\label{sec:notationw2} 

 We recall a notation introduced by Fayers \cite{fayers}, for partitions in $\Bgot_\gamma$. Since $\Bgot_\gamma$ is of weight $2$, a partition $\lambda$ in $\Bgot_\gamma$ is obtained from the abacus of $\gamma$ by moving twice a bead down one position. This can be done in three different ways
 \begin{itemize}
 \item The lowest bead on each of the runners $i$ and $j$, with $i<j$, is moved one position down. In which case $\lambda$ is denoted $\lr{i,j}$.
 \item The lowest bead on runner $i$ is moved down two positions. In this case $\lambda$ is denoted $\lr{i}$.
 \item The lowest bead on runner $i$ is moved down one position and the next bead above it is moved down one position. In this case $\lambda$ is denoted $\lr{i^2}$.
 \end{itemize}
 
So that every partition in $\Bgot_\gamma$ corresponds to exactly one of $\lr{i,j}$ for $0\leq i<j \leq p-1$ or $\lr i$ or $\lr{i^2}$ for $0\leq i \leq p-1$.  We refer to this notation as $\lr{\cdot}$ notation for partitions in $\Bgot_\gamma$. This shows that there is a total of $\binom{p}{2} +2p=\frac{p(p+3)}{2}$ partitions in $\Bgot_\gamma$.

\begin{remark*} The $\lr{\cdot}$ notation is also well defined in blocks which are not self-conjugate. \end{remark*}

\subsubsection{Conjugation in the abacus and notation for self-conjugate partitions}

In order to characterize self-conjugate partitions in $\Bgot_\gamma$ in terms of $\lr{\cdot}$ notation, let us study the form of the abacus of $\gamma$.

Let us refer to the following transformation of a runner as \emph{reversing}: reflecting the runner with respect to a horizontal axis (turning it upside down), and transforming the beads into empty spaces and viceversa.

Since $p$ is odd, any $p$-abacus has a runner which is in the middle, we refer to this runner as the \emph{middle runner}. For any other runner, we can associate what we call its \emph{opposite runner} which is the different runner equally spaced to the middle runner.

Now, conjugation of partitions can be done in the abacus in two steps: first we switch each runner with its opposite runner, and then we reverse all runners simultaneously (with respect to to a same horizontal axis). Thus, the abacus of a self-conjugate $p$-core is such that the middle runner is its own reverse and opposite runners are mutual reverses. For example, the abacus of the self-conjugate $5$-core $\gamma=(6,5,3,2^2,1)$ is 

 \[
\begin{matrix}
4 & 1 & 2 & 3 & 0\\
\bd & \bd & \bd & \bd & \bd\\[-8pt]
\bd & \bd & \bd & \bd &  \pa \\[-8pt]
\bd & \pa & \bd & \bd & \pa  \\[-8pt]
\bd & \pa & \pa & \bd & \pa \\[-8pt]
\bd & \pa & \pa & \pa & \pa \\[-8pt]
\pa & \pa & \pa & \pa & \pa \\[-8pt]
\end{matrix}
\]

The labelling of runners in the abacus defined in \S \ref{sec:partitions} has the property that the labels of two runners opposite to each other, add up to $p-1$. Let us see why. Recall the total order $\lessdot$ defined on the set of runners of the abacus, in \S \ref{sec:partitions}. It is easy to see, following the discussion in the previous paragraph, that conjugation reverses this order, that is, if $R$ and $S$ are two runners of the abacus of a $p$-core such that $R \lessdot S$, and $R'$ and $S'$ are their images under conjugations, then $S' \lessdot R'$. So if a runner $R$ has label $r$, then its image $R'$ under conjugation, has label $p-1-r$. In the abacus of a self-conjugate $p$-core, this means that the labels $(r_0,r_1, \ldots, r_{p-1})$ must satisfy $r_i=p-1-r_{p-1-i}$. See the example above.

Now, observe that any self-conjugate partition in $\Bgot_\gamma$ is obtained from the abacus of $\gamma$ by sliding down two beads, each one in a runner opposite to the other: first, we are saying that such a partition belongs to the first kind of partitions described in \S \ref{sec:notationw2}. Indeed, being self-conjugate, the middle runner is equal to its reverse. A runner on which at most two beads have been moved and which is equal to its reverse is of one of the following forms

 \[
\begin{matrix}
\vdots \\
\bd \\[-8pt]
\bd \\[-8pt]
\bd \\[-8pt]
\pa \\[-8pt]
\pa \\[-8pt]
\pa \\[-8pt]
\pa \\[-8pt]
\vdots
\end{matrix}
\qquad
\text{or}
\qquad
\begin{matrix}
\vdots \\
\bd \\[-8pt]
\bd \\[-8pt]
\bd \\[-8pt]
\pa \\[-8pt]
\bd \\[-8pt]
\pa \\[-8pt]
\pa \\[-8pt]
\vdots
\end{matrix}
\qquad
\text{or}
\qquad
\begin{matrix}
\vdots \\
\bd \\[-8pt]
\bd \\[-8pt]
\pa \\[-8pt]
\bd \\[-8pt]
\pa \\[-8pt]
\bd \\[-8pt]
\pa \\[-8pt]
\vdots
\end{matrix}
\]
but the third option is not possible since it requires more than two movements. The second option requires only one bead movement, but that means that the second (and last) movement occurs in a runner different from the middle runner, which would then not be equal to the reversed opposite runner, since its opposite runner has all of its beads all the way up. This leaves the first option as the only possibility. Thus, none of the two bead movement are done in the middle runner; they are done in a pair of opposite runners. Such a partition is then written $\lr{i,j}$ in the $\lr{\cdot}$ notation, and more precisely, we have:

\begin{lemma}\label{lem:labelselfc}
The self-conjugate partitions in $\Bgot_\gamma$ are exactly
\[
\nu_k=\Lr{\frac{p-1}{2}-k,\frac{p-1}{2}+k} \quad \text{for}\quad 1\leq k\leq \frac{p-1}{2}.
\]
\end{lemma}

\medskip

\begin{example} There are exactly two self-conjugate partitions in the $5$-block $\Bgot_{(2,2)}$, from Example \ref{ex:5core}. They are $(6,3,2,1^3)=\lr{1,3}$ and $(7,2,1^5)=\lr{0,4}$. \demo
\end{example}

\subsubsection{Notation for \texorpdfstring{$p$}{p}-regular partitions in \texorpdfstring{$\Bgot_\gamma$}{Bg}}
For this subsection, $\gamma$ is not necessarily self-conjugate. 

In \cite{fayers}, Fayers introduced a notation $\lceil \cdot \rfloor$ for indexing $p$-regular partitions in $\Bgot_\gamma$. We recall this notation and we slightly modify it for our convenience. Recall that, any partition in $\Bgot_\gamma$ corresponds to one of $\lr{i}$, $\lr{i^2}$, or $\lr{i,j}$ for $0\leq i<j \leq p-1$ and $0\leq i \leq p-1$. Now, some of these partitions are $p$-regular, and this depends on $\gamma$.  Let $(\li i \gamma_j)_{i j}$ be the pyramid of $\Bgot_\gamma$.  For  $0\leq i \leq j <p$ with $i<p-1$, define

\[
\fn{i,j}:= 
\begin{cases*}
\lr{i+1} & if $i=j$ and $\li{i+1}0_{i+2}$ ,\\
\lr{i+1,i+2} & if $i=j$ and $\li{i+1}1_{i+2}$ ,\\
\lr{i+1,j}  & if $i\neq j$ and $\li{i+1}0_{j}$ ,\\
\lr{j^2}  & if $i\neq j$ and $\li{i+1}1_{j}$ and $\li{i}0_{j}$ ,\\
\lr{i}  & if $i\neq j$ and $\li{i}1_{j}$ and $\li{i}0_{j+1}$,\\
\lr{i,j+1}  & if $i\neq j$ and $\li{i}1_{j+1}$ .\\
\end{cases*}
\]

From the abacus and the information encoded in the pyramid, it is easy to see that each $p$-regular partition corresponds to exactly one of $\fn{i,j}$ for $0\leq i \leq j <p$ and $i<p-1$. We refer to this notation for $p$-regular partitions in $\Bgot_\gamma$ as $\fn{\cdot}$-notation.

\begin{remark*}
If $\fn{\cdot}^\textup{F}$ denotes Fayers' labelling for $p$-regular partitions as originally defined in \cite{fayers}, then
\[
\fn{i,j}:= 
\begin{cases*}
\fn{i+1}^\textup{F} & if $i=j$ and,\\
\fn{i+1,j}^\textup{F} & otherwise.
\end{cases*}
\]
\demo
\end{remark*}

\noindent Notice that $p$-regular partitions are  in bijection with the set of all but one entry in $(\li i \gamma_j)$ by making
\begin{equation}\label{eq:bijpyramid}
\begin{array}{ccc}
\Reg{p}{\Bgot_\gamma} & \leftrightsquigarrow & (\li i \gamma_j)\\
\fn{i,j} & \longleftrightarrow & \li i \gamma_j
\end{array}
\end{equation}
for $0\leq i \leq j <p$ and $i<p-1$. The reason for using $\fn{\cdot}$-notation (and shifting the original definition) is the easy description of this correspondence. There is only the entry $\li{p-1} \gamma_{p-1}$, the last entry to the right of row $0$ in the pyramid, which is not associated to a $p$-regular partition.

\begin{example}\label{ex:5core2} Recall the pyramid for the $5$-block $\Bgot_{(2,2)}$ of $\F_5\mathfrak{S}_{14}$, shown in
Example \ref{ex:5core}. The $5$-regular partitions in $\Bgot_{(2,2)}$ are $\Reg{4}{\Bgot_{(2,2)}}=\{ (12,2),\ (7^2),\ (6,4^2),\ (3^4,1^2),\ (11,3),\ (7,4,3),\ (6,3^2,1^2),\ (5, 3^2,1^3),$ $(3^3, 2, 1^3),\ (8, 3^2),\ (7,3^2 , 1),\ (6, 3, 2, 1^3),\ (4,3^2, 1^4),\ (7, 2^2, 1^3) \}$ their corresponding $\fn{\cdot}$ notations, in the same order are $\fn{3,3},\ \fn{2,2},\ \fn{1,1},\ \fn{0,0},\ \fn{3,4},\ \fn{2,3},\ \fn{1,3},\ \fn{1,2},\ \fn{0,1},\ \fn{2,4},\ \fn{1,4},\ \fn{0,3},\ \fn{0,2}$ and $\fn{0,4}$. \demo
\end{example}

%--------------------------------------------------------------------------------------------------------

\subsection{Richards' \texorpdfstring{$\dr$}{dr} map and the Mullineux map in the pyramid}\label{sec:big}

In this section we explore in detail the correspondence (\ref{eq:bijpyramid}) between $p$-regular partitions in $\Bgot_\gamma$ and entries in the pyramid. We will see how these partitions are distributed in the pyramid: in which positions are the self-mullineux partitions and what the position in the pyramid says with respect to the dominance order.

\subsubsection{Richards' \texorpdfstring{$\dr$}{dr} map}

In \cite{richards}, Richards gives an explicit description of decomposition numbers for blocks of weight $2$. For this, he defines a value $\dr \lambda$ associated to every partition $\lambda$ in such a block. We recall this definition and some of his results which we use later.

As already noted, the core $\gamma$ of any partition $\lambda$ in $\Bgot_\gamma$ is obtained by succesive removal of two rim $p$-hooks. Let $\dr \lambda$ be the absolute value of the difference of the leg lengths of the two rim $p$-hooks. Then $0\leq \dr \lambda = \dr \lambda' \leq p-1$. The value of $\dr \lambda$ is independent of the way in which the rim hooks are removed, see \cite[Lemma 4.1]{richards}. Then $\dr \lambda$ is well defined.

Denote by $\dr_l$ the set of partitions in $\Bgot_\gamma$ such that their $\dr$ value is $l$, that is,
\[
\dr_l=\{\lambda \vdash n \mid \lambda \in \Bgot_\gamma \ \ \text{and}\ \ \dr \lambda=l\}.
\]
We write $\dr_{\gamma,l}$ if the block/core needs to be specified and it is not clear from the context. We denote $\dreg_l$ the subset of $p$-regular partition in $\dr_l$, so that $\dreg_l=\Reg{p}{\Bgot_\gamma} \cap \dr_l$.
We have
\[
\Bgot_\gamma = \bigsqcup_{l=0}^{p-1}\; \dr_l.
\]

\begin{example}\label{ex:5core3} Following Example \ref{ex:5core2}, we have, for partitions in the block $\Bgot_{(2,2)}$ of $\F_5 \mathfrak{S}_{14}$: 
\[
\begin{array}{l}
\dr_0 =  \{\hil{(2^2, 1^{10})},\ \hil{(2^7)},\ (3^4, 1^2),\ (6, 4^2),\ (7^2),\ (12, 2)\} \\
\dr_1 =  \{\hil{(2^3, 1^8)}, (3^3, 2, 1^3),\ (5, 3^2, 1^3),\ (6, 3^2, 1^2),\ (7, 4, 3),\ (11, 3)\}\\
\dr_2 = \{\hil{(3^3, 1^5)},\ (4,3^2, 1^4),\ (6, 3, 2, 1^3),\ (7, 3^2, 1),\ (8, 3^2)\} \\
\dr_3 = \{\hil{(6, 3, 1^5)},\ (7, 2^2, 1^3) \} \\
\dr_4 = \{\hil{(7, 2, 1^5)}\}
\end{array}
\]
where highlighted partitions are those which are not $5$-regular. The rest of the partitions are $5$-regular and, in $\fn{\cdot}$-notation they are, in the same order:

\[
\begin{array}{l}
\dreg_0 =  \{\fn{0,0}, \fn{1,1}, \fn{2,2}, \fn{3,3}\} \\[2mm]
\dreg_1 =  \{\fn{0,1}, \fn{1,2}, \fn{1,3}, \fn{2,3}, \fn{3,4}\}\\[2mm]
\dreg_2 = \{\fn{0,2}, \fn{0,3}, \fn{1,4}, \fn{2,4}\} \\[2mm]
\dreg_3 = \{\fn{0,4} \} \\[2mm]
\end{array}
\]

 \demo
\end{example}

From the abacus and the pyramid, we have some information about particular partitions belonging in some of the sets $\dr_l$.

\begin{lemma}[\text{\cite[\S 4, p. 398]{richards}}]\label{lem:labelselfconj} Let $1\leq k \leq \frac{p-1}{2}$ and let $\nu_k = \Lr{\frac{p-1}{2}-k,\frac{p-1}{2}+k}$ be a self-conjugate partition in $\Bgot_\gamma$. Then $\nu_k \in \dr_{2k-1}\cup \dr_{2k}.$
\end{lemma}

This is easily seen in the abacus. We include a proof for completeness, which is adapted to notation in this paper.

\begin{proof}
Consider the abacus of a partition $\lambda$, and suppose that the runners are labelled in some way with integers $0, 1, \ldots, p-1$. Let $1\leq r \leq p-1$, and let us call \emph{runner-hook $r$}, the hook added to $\lambda$ by moving down one position the last bead in runner $r$. Recall, from \S \ref{sec:partitions}, that the leg-length of the runner-hook $r$ is the number of beads between the start and final position of the moved bead. Now, If we take the abacus of $\gamma$, from the chosen labelling of the runners (defined in the same section), the leg-length of the runner-hook $r$ is 
\[
p-1-r.
\]

We want to calculate $\dr \nu_k$. For this, we calculate the leg-lengths of two $p$-hooks successively added to $\gamma$ to obtain $\nu_k$: the partition $\nu_k = \Lr{\frac{p-1}{2}-k,\frac{p-1}{2}+k}$ can be obtained from the abacus of $\gamma$ by first moving down one bead in runner $\frac{p-1}{2}-k$. From the discussion above the leg-length of this hook is
\[
l_1 := p-1-\left(\frac{p-1}{2}-k\right)=\frac{p-1}{2}+k.
\]
Now, we have a new abacus in which we have to calculate the leg-length of the runner-hook $\frac{p-1}{2}+k$, runner in which we move the second bead. In the previous abacus for $\gamma$, this leg-length would have been $p-1-\left(\frac{p-1}{2}+k\right)=\frac{p-1}{2}-k$, but since we already moved one bead in the abacus for $\gamma$, we might have added $1$ to this leg-length, since we might have added one bead (the first moved bead) to the set of beads between the start and final position of the second bead. This happens only if the difference of positions of the last beads in runners $\frac{p-1}{2}-k$ and $\frac{p-1}{2}+k$, in the abacus for $\gamma$, is less than $p$. In other words, only if $\li{(\frac{p-1}{2}-k)}\gamma_{(\frac{p-1}{2}+k)}=1$. Hence, the leg-length of the runner hook $\frac{p-1}{2}+k$ in this new abacus is
\[
l_2 := 
\begin{cases*}
\frac{p-1}{2}-k+1 & if $\li{(\frac{p-1}{2}-k)}\gamma_{(\frac{p-1}{2}+k)}=1$,\\
\frac{p-1}{2}-k & otherwise.
\end{cases*}
\]
Then 
\[
\dr \nu_k:= 
\begin{cases*}
2k-1 & if $\li{(\frac{p-1}{2}-k)}\gamma_{(\frac{p-1}{2}+k)}=1$,\\
2k & otherwise.
\end{cases*}
\]
\end{proof}

The family of partitions $\dr_0$ can be partitioned into two sets $\dr_0^+$ and $\dr_0^-$ as follows. A partition $\lambda$ of weight $2$ has either two rim $p$-hooks or one rim $p$-hook and one rim $2p$-hook (since the $p$-weight is also equal to the number of rim hooks of length divisible by $p$). In the first case, Richards showed that the leg lengths of the rim $p$-hooks are consecutive integers. Define $\lambda$ to be in $\dr_0^+$ or $\dr_0^-$ following

\[
\lambda \in \dr_0 \text{ has }
\begin{cases*}
\;\;\; \text{two rim $p$-hooks} & $ \begin{cases*} \lambda \in \dr_0^+ & if the larger leg is of even length,\\
 \lambda \in \dr_0^- & otherwise. \end{cases*}$\\
 \qquad \qquad \text{or}\\
\begin{array}{l}
\text{ one rim $p$-hook and}\\
\text{ one rim $2p$-hook}
\end{array} & $ \begin{cases*} \lambda \in \dr_0^+ & if the leg length of the rim $2p$-hook is $\equiv 0 \text{ or } 3$ mod $4$, \\
 \lambda \in \dr_0^- & otherwise. \end{cases*}$
\end{cases*} \vspace{1cm}
\]
Richards proved the following

\begin{proposition}[\text{\cite[Lemma 4.2 and 4.3]{richards}}] \label{prop:richardsdrond}
For $0\leq i\leq p-1$, the set $\dr_l$ is totally ordered by $\trianglelefteq$. Moreover the partitions of $\Bgot_\gamma$ which are $p$-singular are precisely: the smallest partition in each $\dr_l$ for $0 < l \leq p-1$ and the smallest partition in each of $\dr_0^+$ and $\dr_0^-$.
\end{proposition}

Then, there are exactly $(p-1)+2=p+1$ $p$-singular partitions in $\Bgot_\gamma$. Which agrees with the counting of partitions and $p$-regular partitions in $\Bgot_\gamma$ in \S \ref{sec:notationw2}.

Richards's $\dr$-function on $p$-regular partitions can be expressed in the $\fn{\cdot}$ notation. Direct analysis with the abacus of $\gamma$ and its relation with the leg-lengths of $p$-hooks of partitions in $\Bgot_\gamma$ gives:

\begin{proposition}[\text{\cite[Proposition 4.1]{fayers}}]\label{prop:drond} Let $(\li i \gamma_j)_{i j}$ be the pyramid of $\gamma$. For $0 \leq i \leq j <p$ and $i<p-1$ 
\[
\dr \fn{i,j}= j-i-1 + \li i \gamma_j.
\] 
\end{proposition}

 With such an expression for $\dr \fn{i,j}$, together with the correspondence (\ref{eq:bijpyramid}), we can say precisely to which positions in the pyramid correspond the partitions $\dreg_l$ for a given $l$. This helps for graphically visualizing $\dreg_l$ in the pyramid:

\begin{corollary} \label{cor:richardscorresp}
For $1\leq l \leq p-1$, the set $\dreg_l$ is in correspondence with the set  of ``1'' entries in the $l$-th row on the pyramid and the ``0'' entries in the $(l+1)$-th row.
The set $\dreg_0$ corresponds to the first $p-1$ entries on the $0$-th row (all entries in the $0$-th row are ``1'') and the ``0'' entries in the $1$-st row.
\end{corollary}

\begin{proof}

Let $1\leq l \leq p-1$, and let $\lambda \in \dreg_l$. Write $\lambda$ as $\lambda=\fn{i,j}$ with $0 \leq i \leq j <p$ and $i<p-1$.

Suppose that $i<j$. If $\li i 0_j$, then from Proposition \ref{prop:drond}, $j-i=l+1$. So that $\li i\gamma_j=0$ is in row $l+1$. Now, if $\li i 1_j$ then from Proposition \ref{prop:drond}, $j-i=l$. So that $\li i\gamma_j=1$ is in row $l$.

If $i=j$, by definition of the pyramid, $\li i 1_i$ and this entry is on the $0$-th row. On the other hand, Proposition \ref{prop:drond} implies that $l=0$. And since $0\leq i < p-1$, this entry is one of the first $p-1$ entries in $0$-th row.

\end{proof}

\begin{example}\label{ex:5core4} We continue Example \ref{ex:5core3}. Associating $p$-regular partitions in $\Bgot_{(2,2)}$ with their corresponding entries in the pyramid, we have that, in the pyramid of $\Bgot_{(2,2)}$, entries are distributed as follows:
\begin{center}
\includegraphics[scale=0.6]{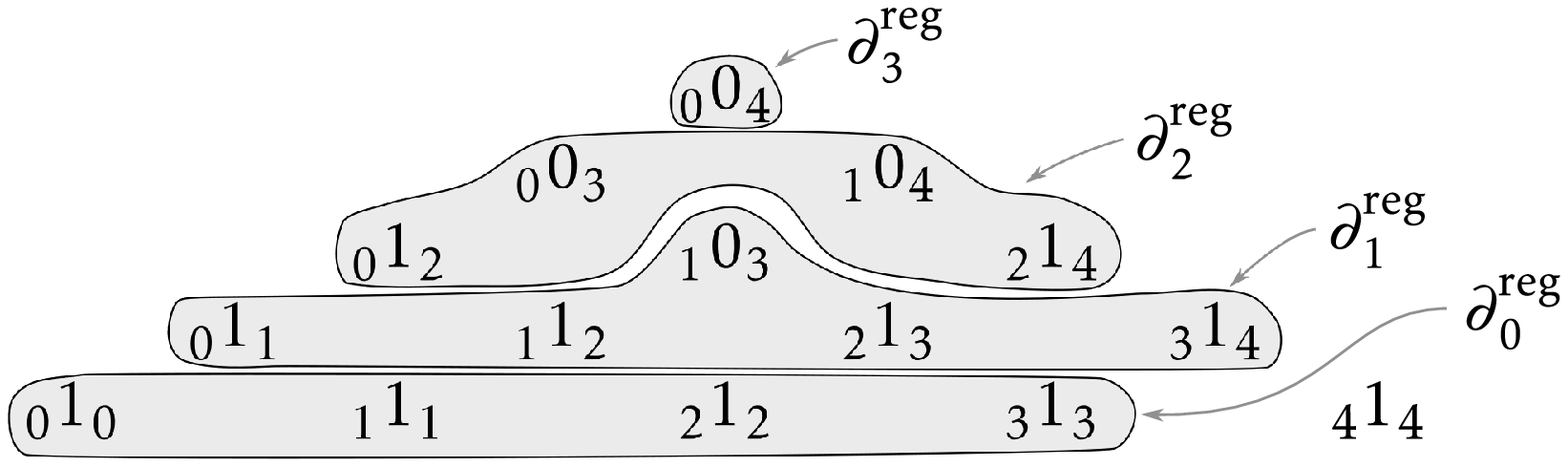}
\end{center}
\demo
\end{example}

%-------------------------------------------------------------------------------------

\subsubsection{Conjugation and the Mullineux map in the sets \texorpdfstring{$\dr$}{dr}}\label{sec:conjugat}
In this section we characterize self-conjugate and self-Mullineux partitions in a given set $\dr_l$, by studying how these involutions behave with respect to $\lessdom$. 

Recall, from Proposition \ref{prop:richardsdrond} that the sets $\dr_l$ are totally ordered by $\lessdom$. Let $1\leq l \leq p-1$ and suppose that $|\dr_l|=k_l$. Denote the partitions in $\dr_l$ as
\[
\hil{\lambda_1} \lessdom \lambda_2 \lessdom \cdots \lessdom \lambda_{k_l-1} \lessdom \lambda_{k_l},
\]
where the highlighted partition is the unique $p$-singular partition. Now, from the fact that $\dr \lambda= \dr \lambda'$ for any partition $\lambda$, and that conjugation reverses the dominance order (\cite[I.1.11]{macdonald}), we get that $\lambda_i'=\lambda_{k_l-i+1}$ for any $1\leq i \leq k_l$. Graphically, conjugation does the following in $\dr_l$ for $1\leq l \leq p-1$:

\begin{center}
\includegraphics[scale=0.27]{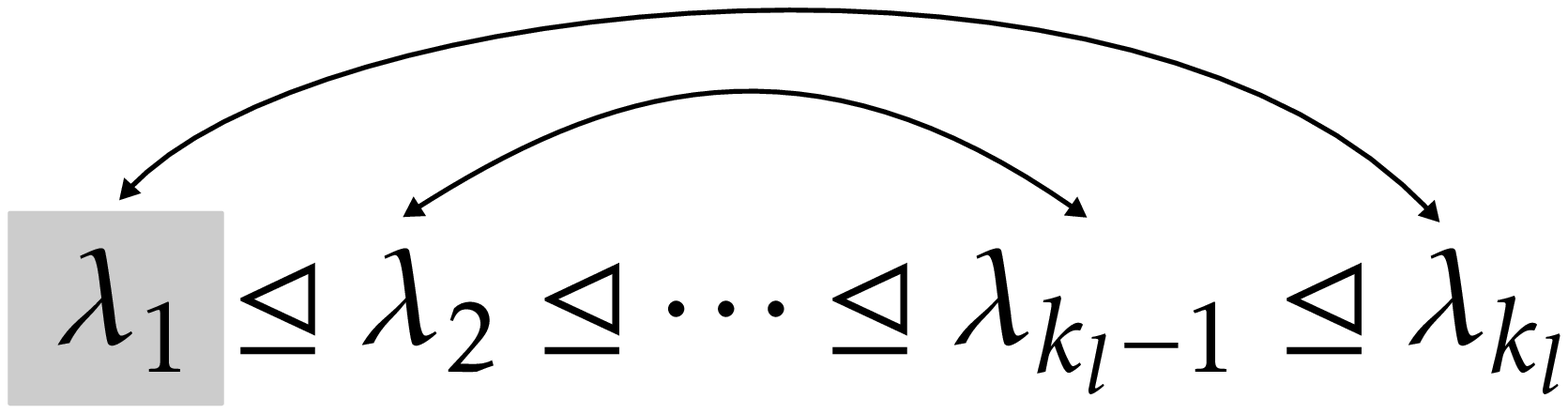}
\end{center}

For $l=0$ it is easy to see that, since $p$ is odd, if $\lambda \in \dr_0^+$ then $\lambda' \in \dr_0^-$, that is, conjugation changes the sign of partitions in $\dr_0$. Then, $k_0$ is even. Let $k_0=2j_0$ and denote partitions in $\dr_0^+$ as 
\[
\hil{\lambda_1^+} \lessdom \lambda_2^+ \lessdom \lambda_3^+ \lessdom \cdots \lessdom \lambda_{j_0-1}^+ \lessdom \lambda_{j_0}^+,
\]
and denote partitions in $\dr_0^-$ as 
\[
\hil{\tau_1^-} \lessdom \tau_2^- \lessdom \tau_3^- \lessdom \cdots \lessdom \tau_{j_0-1}^- \lessdom \tau_{{j_0}}^-,
\]
where the highlighted partitions are the only $p$-singular partitions in $\dr_0$. Then, since conjugation changes the sign and reverses the dominance order $(\lambda_i^+)'=\tau_{j_0-i+1}$. Graphically, conjugation does the following in $\dr_0$:

\begin{center}
\includegraphics[scale=0.27]{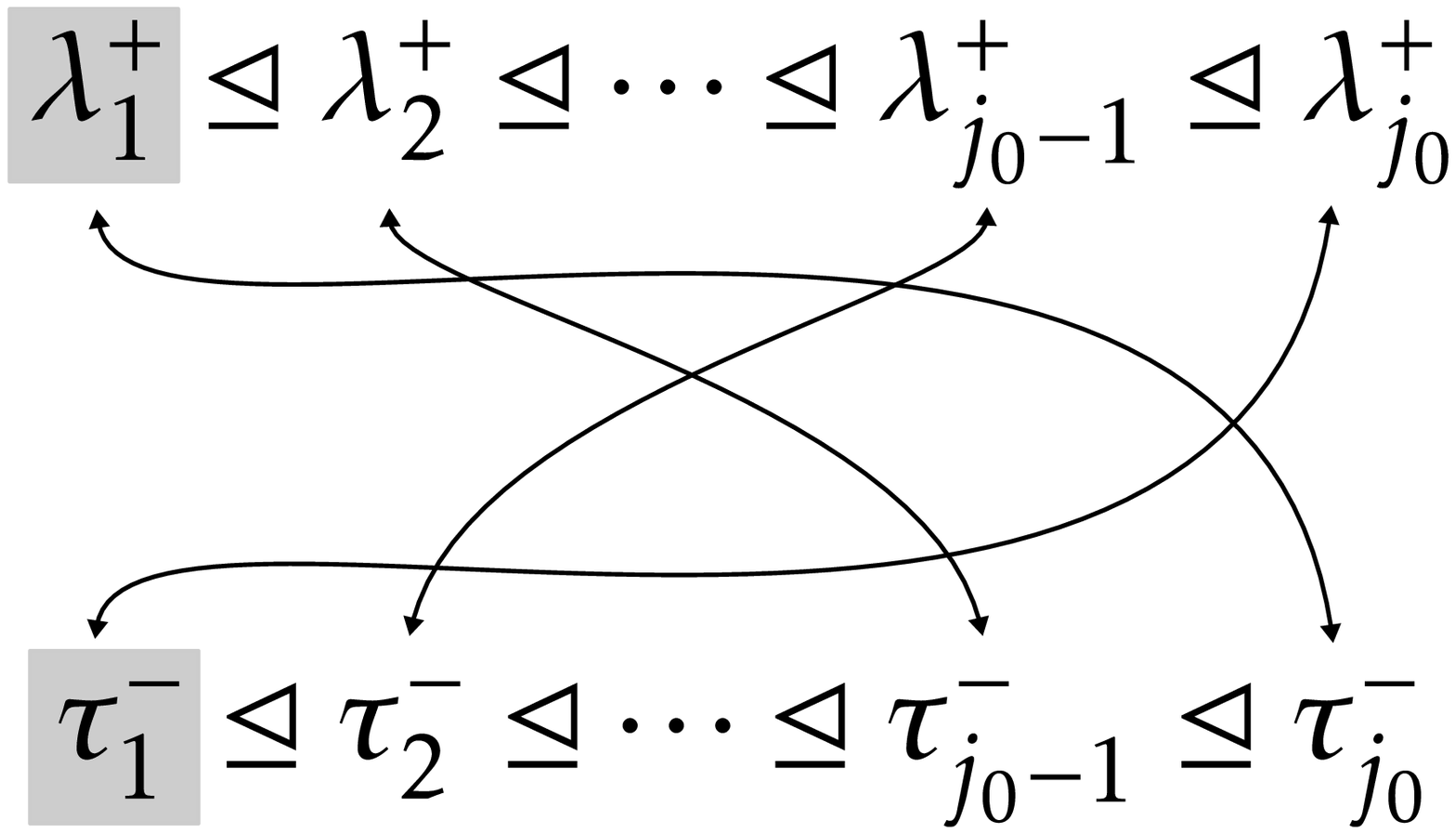}
\end{center}

Let us now see that the Mullineux map behaves in a similar way, but within $\dreg_l$. Recall that, for a $p$-regular partition $\mu$, the $p$-cores of $\mu$ and $\m(\mu)$ are conjugates. Since $\Bgot_\gamma$ is self-conjugate, then for any $\mu \in \Reg{p}{\Bgot_\gamma}$ the partition $\m(\mu)$ is also in $\Reg{p}{\Bgot_\gamma}$. Moreover, Richards showed that $(\m(\mu))'$ is the biggest partition (for $\lessdom$) in $\Bgot_\gamma$ such that $(\m(\mu))'\lessdom \mu$ and $\dr (\m(\mu))'= \dr \mu$, and if $\dr \mu=0$, it has the same sign as $\mu$ (\cite[Th. 4.4 and Prop. 2.12]{richards}). Then, for $2\leq i \leq k_l$ we have that $(\m(\lambda_i))'=\lambda_{i-1}.$ So that $\m(\lambda_i)=\lambda_{i-1}'=\lambda_{k_l-(i-1)+1}=\lambda_{(k_l-i+1)+1}.$ Graphically, the Mullineux map does the following in $\dr_l$ for $1\leq l \leq p-1$:
\begin{center}
\includegraphics[scale=0.3]{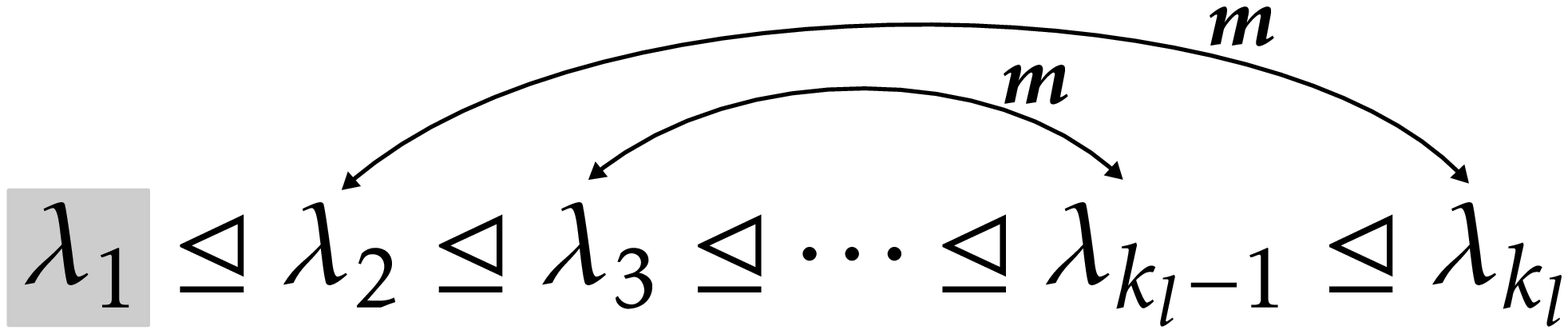}
\end{center}
For $l=0$, and for $2\leq i \leq j_0$ we have that $(\m(\lambda_i^+))'=\lambda_{i-1}^+$. We get that $\m(\lambda_i^+)=(\lambda_{i-1}^+)'=\tau_{j_0-(i-1)+1}^-=\tau_{(j_0-i+1)+1}^-.$ Graphically, the Mullineux map does the following in $\dr_0$:

\begin{center}
\includegraphics[scale=0.3]{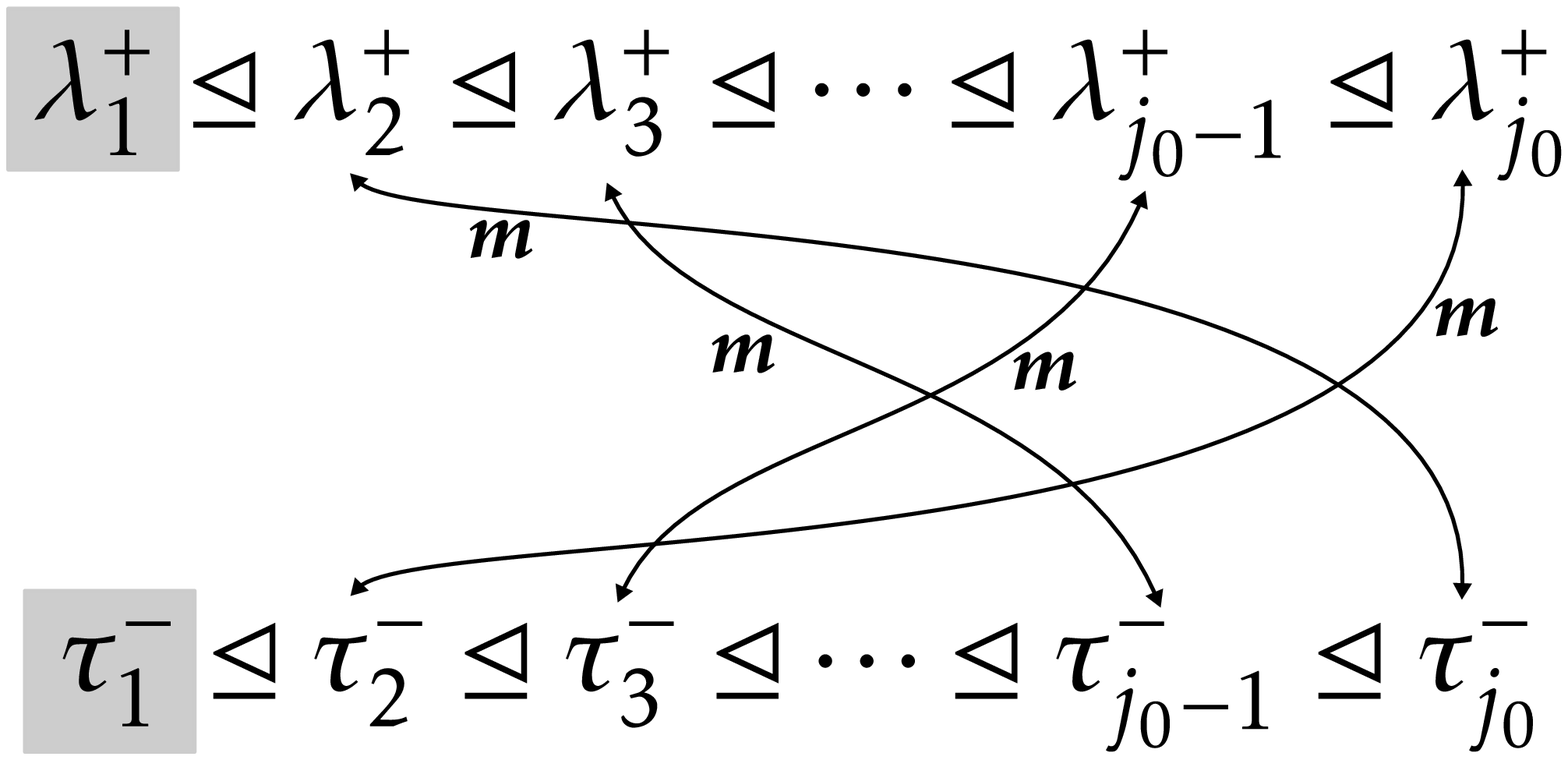}
\end{center}

The behaviour of these two involutions in the sets $\dr_l$ allows to easily deduce in which sets lie their fixpoints, that is, the self-conjugate and the self-Mullineux partitions: it all depends on the parities of $k_l$ for $1\leq 1 \leq p-1$ and $j_0$. The conclusion is summarized in the following lemma.

\begin{lemma}\label{lemma:drblocks}
The set $\dr_0$ does not contain any self-conjugate or self-Mullineux partition and $|\dr_0|$ is even. For $1 \leq l \leq p-1$, the set $\dr_l$ contains \emph{either} exactly one self-conjugate partition or one self-Mullineux partition: if $|\dr_l|$ is even, it contains one self-Mullineux partition. If $|\dr_l|$ is odd, it contains one self-conjugate partition.
\end{lemma}

\begin{proof}
From the discussion above describing conjugation in $\dr_l$ and the Mullineux map in $\dreg_l$, we see that none of these involutions has fixpoints in $\dr_0$. For $l\geq 1$, on the other hand, we can see that if $k_l=|\dr_l|$ is even, then conjugation defines pairs $(\lambda, \lambda')$ with $\lambda \neq \lambda'$. In this case $|\dreg_l|=k_l-1$ is then odd, and the Mullineux map has a unique fixpoint: the partition in the middle of the list $\lambda_2, \lambda_3,\ldots, \lambda_{k_l}$. 
The contrary occurs when $k_l$ is odd.
\end{proof}

\begin{example} \label{ex:5core5}
We continue Example \ref{ex:5core4} to illustrate this lemma. The block $\Bgot_\gamma$ contains exactly two self-conjugate partitions: $(6, 3, 2, 1^3)=\lr{1,3}$ and $(7, 2, 1^5)=\lr{0,4}$; and two self-Mullineux partitions: $(6, 3^1, 1^2)=\fn{1,3}$ and $(7,2^2, 1^3)=\fn{0,4}$. The cardinalities of $\dr_1$, $\dr_2$, $\dr_3$ and $\dr_4$ are respectively $6$, $5$, $2$ and $1$. The two self-conjugate partitions are respectively in $\dr_2$ and $\dr_4$ and the two self-Mullineux partitions are respectively in $\dr_1$ and $\dr_3$.
\demo
\end{example}

\begin{remark}
	The operations of conjugation and the Mullineux map in a self-conjugate block of weight $1$ of $\F_p\Sn$ behave exactly as in each set $\dr_l$ for $1 \leq i \leq p-1$.

	Indeed, let $\mathfrak{D}_\gamma$ be a block of weight $1$ of $\F_p\Sn$. It can be shown that $\mathfrak{D}_\gamma$ contains exactly $p$ partitions $\lambda^0, \lambda^1, \ldots, \lambda^{p-1}$ that can be labelled so that 
	\begin{equation}
	\lambda^0 \ldom \lambda^1 \ldom \cdots \ldom \lambda^{p-1},
	\end{equation} 

	where $\lambda^0$ is the unique $p$-singular partition. Moreover, the decomposition numbers of such a block are well known. See \cite[Exercise 5.10]{mathas}. 
	
	Suppose that $\gamma$ is self-conjugate. From the fact that partitions in this block are totally ordered by the dominance order, we obtain, as for each set $\dr_l$, that ${(\lambda^i)}'=\lambda^{p-i+1}$ for every $1 \leq i \leq p-1$.

	Now, using the fact that the \p regular partitions form a UBS and Proposition \ref{fact:mulli}, from the form of the decomposition matrix of $\mathfrak{D}_\gamma$, it can be shown that $\m(\lambda^i)=\lambda^{p-i}$ for $1\leq i \leq p-1$. If $\gamma$ is not self-conjugate, $\m(\lambda^i)$ is the $(p-i)$-th partition of $\mathfrak{D}_{\gamma'}$.

	Recall, from \S \ref{sec:partitions}, that in the abacus display for a partition, sliding down one bead in a runner amounts to adding a \p rim-hook to the Young diagram. The labelling $\lambda^i$ for partitions in $\mathfrak{D}_\gamma$ is compatible with the chosen labelling for the runners: the partition $\lambda^i$ is obtained by sliding down one bead in the runner $i$ in the abacus display for $\gamma$. Hence, the \p rim-hook has arm-length equal to $i$ (leg length $(p-1)-i$). This allows to easily describe the Mullineux map for partitions with weight $1$ as follows: if $\lambda$ is a partition with weight $1$ and \p core $\gamma$, whose \p rim-hook has arm-length $i$, then $\m(\lambda)$ is the unique weight $1$ partition with \p core $\gamma'$ and \p rim-hook with arm-length $p-i$.
	
	\demo
\end{remark}

%-------------------------------------------------------------------------------------------------------

\subsubsection{Dominance order in the pyramid} 
In this section we study how we can compare $p$-regular partitions in $\Bgot_\gamma$ with respect to $\lessdom$ depending on their positions in the pyramid of $\gamma$. This allows to improve what we already know about the distribution of the sets $\dreg_l$ in the pyramid (Corollary \ref{cor:richardscorresp}), and to identify self-Mullineux partitions in the pyramid.

Let $(\li i \gamma_j)_{i j}$ be the pyramid of $\gamma$. We identify the entry $\li i \gamma_j$ with the corresponding $p$-regular partition $\fn{i,j}$ in $\Bgot_\gamma$. 

\begin{proposition}\label{prop:dominancepyramid} Let $1\leq i \leq j < p$ and $i<p-1$. Then $\fn{i-1,j}\lessdom \fn{i-1,j+1}$ and $\fn{i-1,j}\lessdom \fn{i,j}$. In the correspondence with the pyramid array, graphically, we have the local configuration
		\[
		\begin{array}{ccc}
		&&\fn{i-1,j+1} \\		
		&\lessup & \\
		\fn{i-1,j} && \\
		&\lesdow & \\
		&& \fn{i,j}
		\end{array}
		\]
\end{proposition}

\begin{proof} In \cite{richards}, Richards introduces a notation for partitions in the block, which depends on the entries of the pyramid. With this notation he obtains a complete description of the dominance order $\lessdom$ in the block, see \cite[Lemma 4.4]{richards}. We translate Richards' notation in the notation $\fn{\cdot}$. Every partition in the block corresponds to a pair written $\{s,t\}$ for some $0 \leq s < t \leq 2p$. We do not explain here how to associate a partition to $\{s,t\}$ or viceversa. Lemma \cite[Lemma 4.4]{richards} says

\begin{equation}\label{eq:domrich}
\{s,t\} \domleq \{s',t'\} \quad \text{if and only if} \quad s\leq s'\quad \text{and}\quad t \leq t'.
\end{equation}

Using this, our assertion is easily verified.
\end{proof}

This proposition implies that the distribution of the sets $\dreg_l$ in the pyramid, as we know it from Corollary \ref{cor:richardscorresp}, does not only correspond to some entries of the pyramid as a set, but this distribution also agrees with the dominance order, which then is increasing from left to right in the pyramid.

Furthermore, in our case, since $\gamma$ is self-conjugate then, by construction the pyramid is horizontally symmetrical. Hence, each set $\dreg_l$ occurs in symmetrical positions (is equally distributed with respect to the middle of the pyramid). Then, any self-Mullineux partition $\mu \in \dr_l$ (some $1\leq l \leq p-1$), being in the middle of the list
\[
\lambda_2 \lessdom \cdots \lessdom \lambda_l,
\]
occurs in the middle of the pyramid. And any entry in the middle of the pyramid corresponds, reciprocally to a self-Mullineux partition. We state this fact in the following corollary

\begin{corollary}\label{coro:selfm}
 Let $(\li i \gamma_j)_{i j}$ be the pyramid of $\gamma$. The self-Mullineux partitions in $\Bgot_\gamma$ correspond exactly to the entries on the middle column of the pyramid, except for the one on row $0$. In $\fn{\cdot}$-notation, these partition are:
\[
\mu_k= \left\lceil \frac{p-1}{2}-k,\frac{p-1}{2}+k \right\rfloor \qquad \text{for} \qquad 1\leq k \leq \frac{p-1}{2}.
\]
Moreover $\mu_k \in \dr_{2k-1}\cup\dr_{2k}.$
\end{corollary}

\begin{proof}
There is nothing to prove for the first affirmation: indeed, the indices $i=\frac{p-1}{2}-k$ and $j=\frac{p-1}{2}-k$ for $1\leq k \leq \frac{p-1}{2}$ are those exactly in the middle column of the pyramid by construction.
Let us see that $\mu_k \in \dr_{2k-1}\cup\dr_{2k}.$ Since $\mu_k= \left\lceil \frac{p-1}{2}-k,\frac{p-1}{2}+k \right\rfloor$, then $\mu_k$ is in row $\left( \frac{p-1}{2}+k \right)-\left( \frac{p-1}{2}-k \right)=2k$. If $\li{\frac{p-1}{2}+k} \gamma_{\frac{p-1}{2}-k}=0$, then by Corollary \ref{cor:richardscorresp}, $\mu_k \in \dreg_{2k-1}$, otherwise  $\mu_k \in \dreg_{2k}$.
\end{proof}

\begin{example}\label{ex:5core7}
From the last two results, we can add new information to the diagram in Example \ref{ex:5core4}, namely (some) dominance order relations, and we can also highlight the entries corresponding to the self-Mullineux partitions in $\Bgot_{(2,2)}$:
\begin{center}
\includegraphics[scale=0.6]{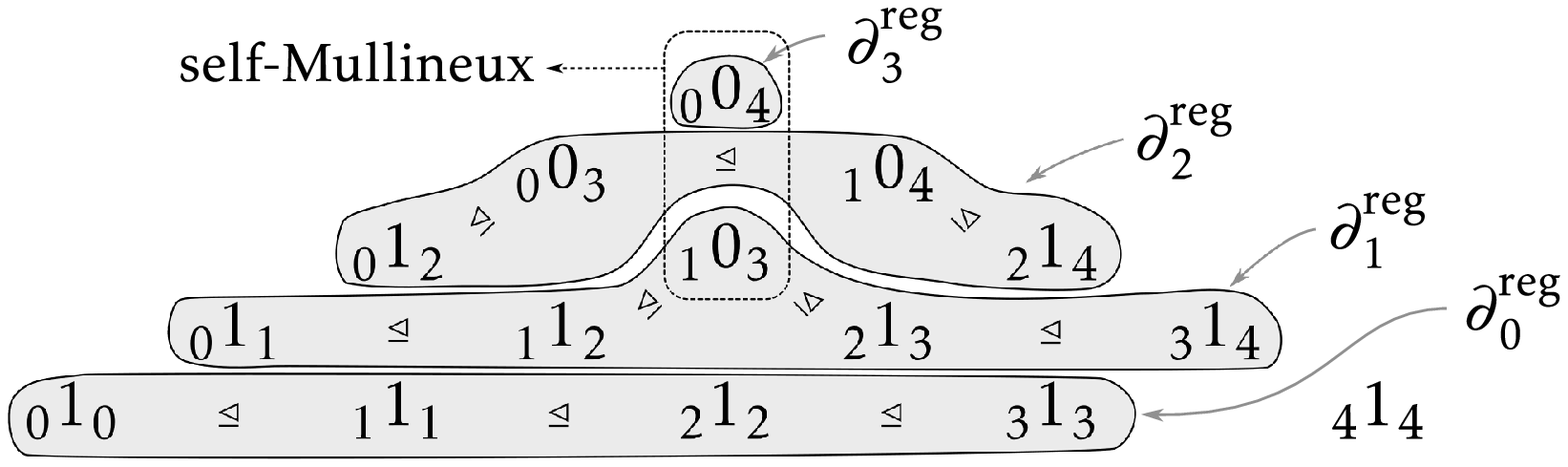}
\end{center}
\demo
\end{example}

%-----------------------------------------------------------------------------------

\subsubsection{Bijection between self-conjugate and self-Mullineux partitions} Here we state the natural correspondance between self-conjugate and self-Mullineux partitions in the block $\Bgot_\gamma$ given by the chosen notations.

A direct consequence of Lemmas \ref{lem:labelselfc}, \ref{lem:labelselfconj}, \ref{lemma:drblocks} and Corollary \ref{coro:selfm} is the following:

\begin{corollary}\label{coro:corresp}
There is a natural one-to-one correspondence between the set of self-conjugate partitions and the set of self-Mullineux partitions in $\Bgot_\gamma$ given by, for $1 \leq k \leq \frac{p-1}{2}$, 
\[
\nu_k=\Lr{\frac{p-1}{2}-k,\frac{p-1}{2}+k} \quad \longleftrightarrow \quad \mu_k=\FN{\frac{p-1}{2}-k,\frac{p-1}{2}+k}.
\]
The partitions $\mu_k$ and $\nu_k$ are the unique self-Mullineux and self-conjugate partitions, respectively, in the set $\dr_{2k-1} \cup \dr_{2k}$.
\end{corollary}

\begin{example}\label{ex:5core6}
We continue Example \ref{ex:5core5}. We have $\{\nu_1,\mu_1\} \subseteq \dr_1 \cup \dr_2$, where $\nu_1=\lr{1,3}=(6, 3, 2, 1^3)$ and $\mu_1=\fn{1,3}=(6, 3^1, 1^2)$; and $\{\nu_2,\mu_2\} \subseteq \dr_3 \cup \dr_4$, where $\nu_2=\lr{0,4}=(7, 2, 1^5)$ and $\mu_2=\fn{0,4}=(7,2^2, 1^3)$. \demo
\end{example}

\begin{remark}
In \cite{bernal}, there is a combinatorial bijection between two families of partitions, depending on an odd prime $p$, of a positive integer $n$; the set of $p$-self-Mullineux partitions of $n$, denoted $\mathcal{M}_p^n$ and the \emph{$p$-BG-partitions} of $n$, denoted $\text{BG}_p^n$. In particular, partitions in $\text{BG}_p^n$ are self-conjugate. This bijection is described combinatorially with the Young diagrams and with some particular arrays of integers. 
From some combinatorial and algebraic observations, it is easy to see that the two sets $\mathcal{M}_p^n$ and $\text{BG}_p^n$ have the same cardinalities. Now, from \cite[Proposition 6.1]{brunatgramain}, it can be seen that for any block $\Bgot$ of $\F_p\Sn$, we have $|\mathcal{M}_p^n \cap \Bgot|=|\text{BG}_p^n \cap \Bgot|$. In words, when restricting to blocks, we also have that the number of $p$-self-Mullineux partitions is equal to the number of $p$-BG-partitions. In a block of weight $2$, the set of $p$-BG-partitions is equal to the set of self-conjugate partitions. On the other hand, it can be shown that the bijection in \cite{bernal} restricts to each block, that is, preserves the $p$-core of a partition. Hence, a question is whether this bijection coincides with the correspondence in Corollary \ref{coro:corresp}. 
We have checked using GAP \cite{gap} that this is indeed the case for $n \leq 56$.
\end{remark}

%----------------------------------------------------------------------------------------

\subsubsection{Two lemmas} In this section we state two technical lemmas which are important for the proof of Theorem \ref{th:main}.

We know more or less to which sets $\dr_l$ the self-conjugate and self-Mullineux partitions in $\Bgot_\gamma$ belong; indeed, for $1\leq l \leq \frac{p-1}{2}$, we know that the pair of partitions $\mu_k$ and $\nu_k$ are in $\dr_{2k-1} \cup \dr_{2k}$, see Corollary \ref{coro:corresp}. The information encoded in the pyramid of the block $\Bgot_\gamma$ allows us to be more precise. For this we state the following lemma.

Consider the set $\mu_1, \mu_2, \ldots,\mu_{\frac{p-1}{2}}$ of self-Mullineux partitions in $\Bgot_\gamma$. The entries of the pyramid corresponding to partitions $\mu_1, \mu_2, \ldots,\mu_{\frac{p-1}{2}}$ are $g_k:=\li i \gamma_j$ with $i=\frac{p-1}{2}-k$ and $j=\frac{p-1}{2}-k$, for $1\leq k \leq \frac{p-1}{2}$ (Corollary \ref{coro:selfm}). Such entries are in the middle column of the pyramid. When $k$ runs from $1$ to $\frac{p-1}{2}$, these entries run from bottom to top in the pyramid, and because of the definition and properties of the pyramid, the sequence $g_1, g_2, \ldots,g_{\frac{p-1}{2}}$ is of one the forms $1, 1, \ldots, 1, 0, 0, \ldots, 0$, or $1, 1, \ldots, 1$, or else $0,0,\ldots,0$.
\medskip

\begin{definition}\label{def:delta} In the notation from the preceding paragraph, define $\delta=\delta(\gamma)$ as 
\[ \delta =
\begin{cases*}
0 & if $g_k=0$ for all $1\leq k \leq \frac{p-1}{2}$ ,\\
\textup{max}\ \{k \mid g_k=1\} & otherwise.
\end{cases*}
\]
%i & if $g_k=1$ for all $1 \leq k \leq i$.
\end{definition} 

\medskip

\begin{lemma}\label{lemma:drondb}
For $k=1, \ldots, \frac{p-1}{2}$ we have:
\begin{itemize}
\item If $k \leq \delta$, then $\mu_k \in \dr_{2k}$ and  $\nu_k \in \dr_{2k-1}$;
\item if $k > \delta$, then $\mu_k \in \dr_{2k-1}$ and $\nu_k \in \dr_{2k}$.
\end{itemize}
In a table:
\[
\begin{array}{ccc}
\nu_\frac{p-1}{2} & \in & \dr_{p-1} \\
\mu_\frac{p-1}{2} & \in & \dr_{p-2} \\
\hline
&  \vdots &\\
\hline
\nu_{\delta+1} & \in & \dr_{2(\delta+1)} \\
\mu_{\delta+1} & \in & \dr_{2(\delta+1)-1} \\
\hline
\mu_\delta & \in & \dr_{2\delta} \\
\nu_\delta & \in & \dr_{2\delta-1} \\
\hline
&  \vdots &\\
\hline
\mu_2 & \in & \dr_4 \\
\nu_2 & \in & \dr_3 \\
\hline
\mu_1 & \in & \dr_2 \\
\nu_1 & \in & \dr_1
\end{array}
\]
\end{lemma}

\begin{example}
We continue Example \ref{ex:5core6}. The pyramid is in Example \ref{ex:5core7}. We see that the entries in the middle column corresponding to self-Mullineux partitions are all equal to $0$. Hence, here $\delta=0$, and we have indeed that $\mu_1 \in \dr_1$ and $\mu_2 \in \dr_3$, as this lemma implies.
\demo
\end{example}

\begin{proof}[Proof of Lemma \ref{lemma:drondb}]
The entries $g_1, g_2, \ldots, g_\frac{p-1}{2}$ are exactly those in the middle column of the pyramid (except that on row $0$ which corresponds to a $p$-regular partition in $\dr_0$), from bottom to top and they correspond to the self-Mullineux partitions on the block. These entries are respectively in rows $2, 4, 6, \ldots, p-1.$

Let $1\leq k \leq \frac{p-1}{2}$. If $k > \delta$, then $g_k=0$. Hence $\mu_k$ which is in row $2k$, belongs to $\dr_{2k-1}$, by Corollary \ref{cor:richardscorresp}. Thus, for $k> \delta$, we have $\nu_k \in \dr_{2k}$.
If $k \leq \delta$, then  $g_k=1$, and since $\mu_k$ is in row $2k$ we know that $\mu_k \in \dr_{2k}$. Then, for $k \leq \delta$ we have $\nu_k \in \dr_{2k-1}$.
\end{proof}

The following is a technical lemma concerning the dominance order in $\Reg{p}{\Bgot_\gamma}$ with respect to positions in the pyramid. This is an adaptation of \cite[Lemma 4.4]{richards}, which is a characterisation of the dominance order in the block to our notation for partitions. This lemma is a key fact in the proof of Theorem \ref{th:main}. 

Let $(\li i \gamma_j)_{i j}$ be the pyramid of the $p$-core $\gamma$. For two entries $\li i \gamma_j$ and $\li k \gamma_l$, we say that $\li i \gamma_j$ \emph{is to the left of} $\li k \gamma_l$ (or, equivalently $\li k \gamma_l$ is to the right of $\li i \gamma_j$) if $\li i \gamma_j$ is in a column of the pyramid to the left of the column of $\li k \gamma_l$. In terms of indices, this is equivalent to $i+j < k+l$.

\begin{lemma}[\text{\cite[Lemma 4.4]{richards}}]\label{lem:pyramorder} Let $\lambda, \tau \in \Reg{p}{\Bgot_\gamma}$ such that $\lambda$ is to the left of $\tau$.  Then, $\lambda \trianglelefteq \tau$ or $\lambda$ and $\tau$ are not comparable for the dominance order; written equivalently as $\tau \ntriangleleft \lambda$.
\end{lemma}

\begin{proof}
By definition the pyramid of $\gamma$ has $2p-1$ columns. Suppose that $\lambda$ is in column $c$ and $\tau$ is in column $c+N$. The proof is by induction on $N$. The crucial part is the base case. 
Suppose that $N=1$, and let $\lambda=\fn{i,j}$ in $\fn{\cdot}$-notation. Making correspond positions in the pyramid to $p$-regular partitions the local configuration in column $c+1$ is as follows:
\begin{center}

\begin{tikzpicture}
\draw (0,0) node{$\begin{array}{ccc}
c && c+1\\
\hline
&& \vdots \\
&& \fn{i-2,j+3} \\
\\
&& \fn{i-1,j+2} \\
\\
\vdots && \hil{\fn{i,j+1}} \\
\lambda=\fn{i,j} && \\
\vdots && \hil{\fn{i+1,j}} \\
\\
 && \fn{i+2,j-1} \\
\\
&& \fn{i+3,j-2} \\
&& \vdots

\end{array}$};

\draw [decorate,decoration={brace,amplitude=10pt},xshift=0.5cm,yshift=-0.5pt] 
(2,3) --(2,1) node [black,midway,xshift=1cm]
{\footnotesize $A$};

\draw [decorate,decoration={brace,amplitude=10pt},xshift=0.5cm,yshift=-0.5pt] 
(2,-1.5) --(2,-3.5) node [black,midway,xshift=1cm]
{\footnotesize $B$};

\end{tikzpicture}
\end{center}

Now, if $\tau=\fn{i,j+1}$ or $\tau=\fn{i+1,j}$ (highlighted in the diagram above), Proposition \ref{prop:dominancepyramid} says that $\lambda \lessdom \tau$, so that $\tau \ntriangleleft \lambda$. Let us split the rest of the partitions (or entries of the pyramid) in this column in two sets $A$ and $B$, where 
\[
A = \{ \fn{i-k, j+k+1} \ \mid\  k\geq 1\},
\]
and 
\[
B= \{ \fn{i+k+1, j-k} \ \mid\  k\geq 1\},
\]
where $k$ takes values such that partitions in $A$ and $B$ lie in the pyramid. Consider the two possible cases $\li i \gamma_j=0$ or $\li i \gamma_j=1$. In the first case, we necessarily have $\li i \gamma_{j+1}=0$, as well as all the entries in $A$. In the second case $\li{i+1} \gamma_{j}=1$, as well as all the entries in $B$. In the pyramid, these configurations look as follows:

\begin{center}
\begin{tikzpicture}
\draw (0,3) node{\underline{Case 1:}};

\draw (0,0) node{$\begin{array}{ccc}
c && c+1\\
\hline
&& 0 \\
&& \vdots \\
&& 0 \\
\vdots && 0 \\
\li i \gamma_j=0 && \\
\vdots && * \\
&& \vdots \\
&& * \\
\end{array}$};

\draw(5,3) node{\underline{Case 2:}};

\draw (5,0) node{$\begin{array}{ccc}
c && c+1\\
\hline
&& * \\
&& \vdots \\
\vdots && *  \\
\li i \gamma_j=1 && \\
\vdots && 1 \\
 && 1 \\
&& \vdots \\
&& 1 \\
\end{array}$};
\end{tikzpicture}
\end{center}

Let us consider these two cases. In each case, we translate $\fn{\cdot}$-notation of $\lambda$ and of $\tau$ for each partition $\tau$ in column $c+1$, into $\{\cdot\}$-notation to get $\{s,t\}$ and $\{s', t'\}$, respectively and we will see that $\tau \ntriangleleft \lambda$ by noticing that $\{s', t'\} \lessdom \{s,t\}$ or they are not comparable, using (\ref{eq:domrich}) as in the proof of Proposition \ref{prop:dominancepyramid}. Consider all possible cases for positions of $\lambda$ and $\tau$. We get the following values for $\lambda=\{s,t\}$ and $\tau=\{s',t'\}$, where $k \geq 1$:\\

\begin{minipage}[t]{0.5\textwidth}
\begin{tabular}{cc|cc}
\multicolumn{4}{l}{\underline{Case 1:}} \\
\\
\multicolumn{2}{c|}{$\lambda$} & \multicolumn{2}{c}{$\tau$} \\
\hline
$s$ & $t$ & $s'$& $t'$\\
\hline
$2i+3$ & $2j+1$ & $2i-2k+3$ & $2j+2k+3$ \\
$2i+2$ & $2j+1$ & $2i+2k+5$ & $\leq 2p$ \\
&& $2i+2k+5$ & $2j-2k+5$ \\
&& $2i+2k+5$ & $2j-2k+1$ \\
&& $2i+2k+4$ & $2j-2k+1$ \\
&& $2i+2k+3$ & $2j-2k+2$ \\
&& $2i+2k+3$ & $2j-2k+3$ \\
\end{tabular}
\end{minipage}
\begin{minipage}[t]{0.5\textwidth}
\begin{tabular}{cc|cc}
\multicolumn{4}{l}{\underline{Case 2:}} \\
\\
\multicolumn{2}{c|}{$\lambda$} & \multicolumn{2}{c}{$\tau$} \\
\hline
$s$ & $t$ & $s'$& $t'$\\
\hline
$2i+1$ & $2j+3$ & $2i+2k+5$ & $2j-2k+3$ \\
$2i+2$ & $2j+2$ & $2i+2k+5$ & $\leq 2p$ \\
&& $2i+2k+3$ & $2j-2k+3$ \\
&& $2i-2k+3$ & $2j+2k+3$ \\
&& $2i-2k+2$ & $2j+2k+3$ \\
&& $2i-2k+1$ & $2j+2k+4$ \\
&& $2i-2k+1$ & $2j+2k+5$ \\
\end{tabular}
\end{minipage}
\\ \vspace{0.5 cm}

For each of these possible values for $s,s',t, t'$ we always obtain that $s<s'$ or $t<t'$, then either $\lambda \trianglelefteq \tau$ or $\lambda$ and $\tau$ are not comparable for $\trianglelefteq$, by (\ref{eq:domrich}). That is $\tau \ntriangleleft \lambda$. This concludes the base case.

For the inductive step, suppose that $\tau$ is in column $c+N$, with $N>1$. And suppose that $\tau \trianglelefteq \lambda$. Let us see that there is a contradiction. By Proposition \ref{prop:dominancepyramid}, there is a partition $\tilde{\tau}$ in column $c+(N-1)$ such that $\tilde{\tau}\lessdom\tau$. Since $\tilde{\tau}$ is in column $c+(N-1)$, by induction $\tilde{\tau} \ntriangleleft \lambda$. But $\tilde{\tau}\lessdom\tau$ and $\tau \trianglelefteq \lambda$ imply that $\tilde{\tau} \trianglelefteq \lambda$, a contradiction.
\end{proof}

\subsection{SUBS for \texorpdfstring{$\Bgot_\gamma$}{Bg}}\label{sec:decmat}

This section contains the main result. Recall the set $U_{p,n}$ from \ref{eq:UBSpn} in \S \ref{sec:oddweig}:
\[
U_{p,n}=\{\lambda \in \Reg{p}{n} \mid \m(\lambda)<\lambda\}\ \ \sqcup \  \{\lambda' \mid \lambda \in \Reg{p}{n} \ \text{and} \  \m(\lambda)<\lambda\} \ \sqcup \ \{\lambda \in \Reg{p}{n} \mid \m(\lambda)=\lambda\},
\]
which, together with an order on $\Par{n}$ and a bijection with $\Reg{p}{n}$, form a UBS for $\mathbb{F}_p\Sn$. And which by restriction, for each block $\Bgot_\gamma$,
\[
U_\gamma=\{\lambda \in \Reg{p}{\Bgot_\gamma} \mid \m(\lambda)<\lambda\}\ \ \sqcup \  \{\lambda' \mid \lambda \in \Reg{p}{\Bgot_\gamma} \ \text{and} \  \m(\lambda)<\lambda\} \ \sqcup \ \{\lambda \in \Reg{p}{\Bgot_\gamma} \mid \m(\lambda)=\lambda\},
\]
forms a UBS for this block. We define a new set $V_\gamma$ obtained from $U_\gamma$ by replacing the set of self-Mullineux partitions in $\Bgot_\gamma$ by the set of self-conjugate partitions in that block:
\begin{equation}\label{eq:basicblock}
V_\gamma:=\{\lambda \in \Reg{p}{\Bgot_\gamma} \mid \m(\lambda)<\lambda\}\ \ \sqcup \  \{\lambda' \mid \lambda \in \Reg{p}{\Bgot_\gamma} \ \text{and}\  \m(\lambda)<\lambda\} \ \sqcup \ \{\lambda \in \Bgot_\gamma \mid \lambda=\lambda'\}.
\end{equation}

We denote each of these three subsets as $V_\gamma^1$, $V_\gamma^2$, and $V_\gamma^3$, respectively. The set $V_\gamma$ is in bijection with $\Reg{p}{\Bgot_\gamma}$. Indeed, write $\Reg{p}{\Bgot_\gamma}$ as
\begin{equation}\label{eq:simpleblock}
\Reg{p}{\Bgot_\gamma} =\{\lambda \in \Reg{p}{\Bgot_\gamma} \mid \m(\lambda)<\lambda\}\ \ \sqcup \  \{\m(\lambda) \mid \lambda \in \Reg{p}{\Bgot_\gamma} \ \text{and}\  \m(\lambda)<\lambda\} \ \sqcup \ \{\lambda \in \Bgot_\gamma \mid \lambda=\m(\lambda)\},
\end{equation}
where we denote these three subsets as  $W_\gamma^1$, $W_\gamma^2$, and $W_\gamma^3$, respectively. Note that $V_\gamma^3=\{\nu_k \ \mid \ 1 \leq k \leq \frac{p-1}{2}\}$ and $W_\gamma^3=\{\mu_k \ \mid \ 1 \leq k \leq \frac{p-1}{2}\}$. Hence there is a bijection 
\[
\begin{array}{cccl}
\Psi_\gamma: & V_\gamma & \longrightarrow & \Reg{p}{\Bgot_\gamma} \\
\\
& \lambda & \longmapsto & \begin{cases*}
\lambda & if $\lambda \in V_\gamma^1$,\\
\m(\lambda') & if $\lambda \in V_\gamma^2$,\\
\mu_k & if $\lambda \in V_\gamma^3$ and $\lambda=\nu_k$ for some $1 \leq k \leq \frac{p-1}{2}$,
\end{cases*}
\end{array}
\]
which restricts to the bijection from Corollary \ref{coro:corresp} on the self-conjugate partitions. Let us define a total order in $\Bgot_\gamma$. First, label partitions in $V_\gamma^1=\{\lambda_1, \lambda_2, \ldots, \lambda_t\}$, where $t=|V_\gamma^1|$, in such a way that $\lambda_1 > \lambda_2 > \cdots > \lambda_t$ in the lexicographic order. Now, let  $\prec$ be a total order in $\Bgot_\gamma$ such that 
\begin{equation}\label{eq:tot}
\lambda_1 \succ \lambda_2 \succ \cdots \succ \lambda_t \succ \lambda_1' \succ \lambda_2' \succ \cdots \succ \lambda_t' \succ \nu_1 \succ \nu_2 \succ \cdots \succ \nu_{\delta-1} \succ \nu_\delta \succ \nu_\frac{e-1}{2} \succ  \nu_\frac{e-3}{2} \succ \cdots \succ \nu_{\delta+2} \succ \nu_{\delta+1},
\end{equation}
where the number $\delta$ is as in Definition \ref{def:delta}, and such that for any other partition $\lambda \in \Bgot_\gamma \setminus V_\gamma$, we have $\lambda \prec \tau$ for every $\tau \in V_\gamma$. Having defined the bijection $\Psi_\gamma$ and the total order $\prec$, we can now state our main result:

\begin{theorem}\label{th:main}  The set $(V_\gamma, \prec, \Psi_\gamma)$ is a stable unitriangular basic set (SUBS) for the block $\Bgot_\gamma$. 
\end{theorem}

We state three important facts about decomposition numbers for the proof:

\begin{proposition}[Corollary 4.17 \text{\cite{mathas}}]\label{fact:unitri}
Let $\mu \in \Reg{p}{n}$. Then
\begin{itemize}
\item $d_{\mu\mu}=1$, and 
\item if $d_{\lambda\mu}\neq 0$ then $\lambda \domleq \mu.$
\end{itemize}
\end{proposition}

\begin{proposition}[The Mullineux map]\label{fact:mulli} Let $\lambda \in \Par{n}$ and $\mu \in \Reg{p}{n}$. Then $d_{\lambda\mu}=d_{\lambda' \m(\mu)}$, where $\m$ is the Mullineux map.
\end{proposition}

\begin{proposition}[\text{Theorem 4.4 \cite{richards}}]\label{prop:matrixrich}
Let $p\neq2$ be a prime. Let $\gamma \vdash (n-2p)$ be a $p$-core and $\Bgot_\gamma$ the corresponding block of $\Sn$. Let $\lambda \in \Bgot_\gamma$ and $\mu \in \Reg{p}{\Bgot_\gamma}$. Then $d_{\lambda\mu}=1$ if $\lambda=\mu$ or $\lambda=\m(\mu)'$ or both $\m(\mu)' \lessdom \lambda \lessdom \mu$ and $\dr \lambda - \dr \mu = \pm 1$; otherwise $d_{\lambda\mu}=0$.
\end{proposition}

\begin{proof}[Proof of Theorem \ref{th:main}]

The fact that properties (A) and (B) from Definition \ref{def:SUBSblock}, hold for $V_\gamma$ are a direct consequence of its definition. Property (A) is true by definition of the set $V_\gamma$. Property (B) holds because, in a $p$-block of weight $2$, every self-conjugate partitions is a BG-partition. Let us see this. Recall that, under a certain convention for defining the $p$-quotient of a partition in a block, if $q_{p,\gamma}(\lambda)=(\lambda^{(1)}, \lambda^{(2)}, \ldots, \lambda^{(p)})$ is the $p$-quotient of a partition with $p$-core $\gamma$ then $q_{p,\gamma'}(\lambda')=(\lambda^{(p)}{}', \lambda^{(p-1)}{}', \ldots, \lambda^{(1)}{}')$ is the $p$-quotient of its conjugate partition $\lambda'$. The BG-partitions in a $p$-block are those self-conjugate partitions $\nu$ for which the $(\frac{p+1}{2})$-th partition in the quotient is the empty partition. That is, partitions $\nu$ such that the $p$-quotient is of the form
\[
q_{p,\gamma}(\nu)=(\nu^{(1)}, \nu^{(2)},\ldots, \nu^{(\frac{p-1}{2})}, \emptyset ,\  \nu^{(\frac{p-1}{2})}{}',\ldots, \nu^{(2)}{}',\lambda^{(1)}{}' ).
\] 
In a block of $p$-weight $2$, the quotient $q_{p,\gamma}(\lambda)$ of a self-conjugate partition $\lambda$ is a $p$-multi-partition of total rank $2$. Since $\lambda$ is self-conjugate, $q_{p,\gamma}(\lambda)$ is completely determined by $\lambda^{(1)}, \lambda^{(2)}, \ldots, \lambda^{(\frac{p+1}{2})}$, where either $(\lambda^{(1)}, \lambda^{(2)}, \ldots, \lambda^{(\frac{p-1}{2})})$ is a multipartition of $1$ and $\lambda^{(\frac{p+1}{2})}=\emptyset$, or $\lambda^{(i)}=\emptyset$ for all $1\leq i \leq \frac{p-1}{2}$ and $\lambda^{(\frac{p+1}{2})}$ is a self-conjugate partition of $2$. The second option is not possible, then there exists $1\leq j < \frac{p+1}{2}$ such that $\nu^{(j)}=(1)$ and $\nu^{(j)}=(1)$ and $\nu^{(i)}=\emptyset$ for $i \neq j$ with $1\leq i\ \leq \frac{p+1}{2}$. Hence $\lambda$ is a BG-partition.
 
It remains to prove that $(V_\gamma, \prec, \Psi_\gamma)$ is a unitriangular basic set for the block $\Bgot_\gamma$. For this, consider the square matrix $\widetilde{\mathbf{D}}_\gamma$ formed by the rows of $\mathbf{D}_\gamma$ indexed by $V_\gamma$, arranged according to the total order $\prec$.

\begin{center}
\begin{tikzpicture}\label{eq:matxi}
\draw (0,0) node{
$
\begin{array}{c|ccccccccc}
& \lambda_1\phantom{--} & \cdots &  \phantom{--}\lambda_t\phantom{-} & \m(\lambda_1) & \cdots & \m(\lambda_t) & \mu_1 & \cdots & \mu_{\delta+1}\\
\hline
\lambda_1 & \\[10pt]
\vdots & & \mathbf{D}_1 & & &  \mathbf{D}_2 \\[10pt]
\lambda_t \\[15pt]
\lambda_1' \\[10pt]
\vdots & & \mathbf{D}_3 & & &  \mathbf{D}_4 \\[10pt]
\lambda_t' \\[10pt]
\nu_1 \\[5pt]
\vdots & &&&&&&& \mathbf{D}_6 \\[5pt]
\nu_{\delta +1}
\end{array}$};
\draw (-4.35,0.9)--(2.8,0.9);
\draw (-4.35,-1.8)--(5.5,-1.8);
\draw (-1,3.5)--(-1,-1.8);
\draw (2.8,3.5)--(2.8,-4);
\draw (4,0.9) node{$\mathbf{D}_5$};
\draw (-1, -3) node{* * *};
\end{tikzpicture}
\end{center}

We will show that $\widetilde{\mathbf{D}}_\gamma$ is lower unitriangular. We do it by steps: first, for the square submatrices $\mathbf{D}_1, \ldots, \mathbf{D}_4$, we show that $\mathbf{D}_1=\mathbf{D}_4$ is lower unitriangular and that $\mathbf{D}_2=\mathbf{D}_3=(0)_{t \times t}$. We show as well that $\mathbf{D}_5=(0)_{2t \times \frac{p-1}{2}}$. Finally we show that $\mathbf{D}_6$ is lower unitriangular. Having shown this we will have that $(\tilde{V}_\gamma, \prec, \Psi_\gamma)$ is a stable unitriangular basic set for the block $\Bgot_\gamma$. \\

\noindent \underline{$\mathbf{D}_1$ and $\mathbf{D}_4$:}\  since $\lambda_i$ is $p$-regular for $1 \leq i \leq t$, then $d_{\lambda_i\lambda_i}=1$ for $1 \leq i \leq t$. Let $1 \leq i, j \leq t $. By definition of $\prec$, if $\lambda_i \succ \lambda_j$, then $\lambda_i > \lambda_j$ for the lexicographic order. Hence either $\lambda_i \trianglerighteq \lambda_j$ or $\lambda_i$ and $\lambda_j$ are not comparable for $\trianglelefteq$. By Proposition \ref{fact:unitri}, $d_{\lambda_i\lambda_j}=0$. This shows that $\mathbf{D}_1$ is lower unitriangular. For $1 \leq i,j \leq t$, by Proposition \ref{fact:mulli}, $d_{\lambda_i'\m(\lambda_j)}= d_{\lambda_i\lambda_j}$. Hence $\mathbf{D}_4=\mathbf{D}_1$ is lower unitriangular.\\

\noindent \underline{$\mathbf{D}_2$, $\mathbf{D}_3$, and $\mathbf{D}_5$:} for studying $\mathbf{D}_2$, $\mathbf{D}_3$ and $\mathbf{D}_5$, notice that we have the following property: let $\lambda, \tau \in W_\gamma^1$ such that $\dr \lambda, \dr \tau \geq 1$ and let $\mu \in W_\gamma^3$. Then $\lambda \ntriangleleft \mu$ and $\lambda \ntriangleleft \m(\tau)$.

We prove this property. Let $1\leq l \leq p-1$. From \S \ref{sec:conjugat} we know that, depending on $|\dreg_l|$, the partitions in $\dreg_l$ can be listed as either
\[
\tau_1 \triangleright \tau_2 \triangleright \cdots \triangleright \tau_r \triangleright \mu \triangleright \m(\tau_r) \triangleright \cdots \triangleright \m(\tau_2) \triangleright \m(\tau_1),
\]
or
\[
\tau_1 \triangleright \tau_2 \triangleright \cdots \triangleright \tau_r \triangleright \m(\tau_r) \triangleright \cdots \triangleright \m(\tau_2) \triangleright \m(\tau_1),
\]
where $r=\left\lfloor\frac{|\dreg_l|}{2}\right\rfloor$, $\m$ is the Mullineux map and $\mu$ is some self-Mullineux partition.

On the other hand, by Corollaries \ref{cor:richardscorresp}, \ref{coro:selfm} and Proposition \ref{prop:dominancepyramid}, we know that these partitions are distributed in the pyramid from left to right in increasing dominance order, horizontally symmetrical and that the self-Mullineux partitions are in the middle column. These observations allow to identify three zones in the pyramid with the intersection of $\bigcup_{l \geq 1}\dreg_l$ with the three subsets $W_\gamma^1$, $W_\gamma^2$ and $W_\gamma^3$  of $\Reg{p}{\Bgot_\gamma}$ (defined in (\ref{eq:simpleblock})): $W_\gamma^1$ is the left half, $W_\gamma^2$ is the right half and $W_\gamma^3$ is the column in the middle:

\begin{center}
\includegraphics[scale=0.5]{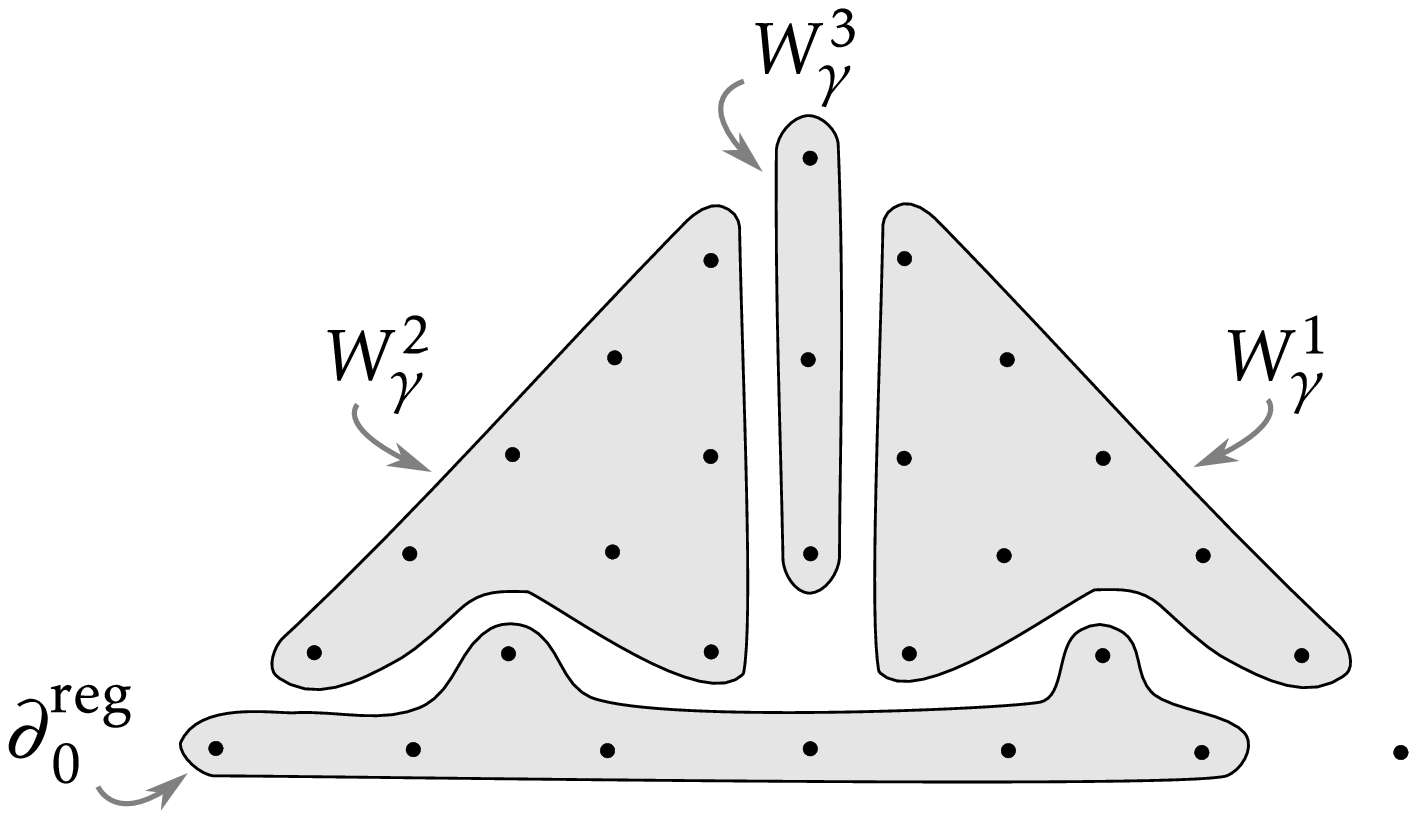}
\end{center}

Now, let $\lambda, \tau \in W_\gamma^1$ with $\dr \lambda , \dr \tau \geq 1$ and let $\mu \in W_\gamma^3$. Then $\mu$ is to the left of $\lambda$ in the pyramid, and $\m(\tau)$ is to the left of $\lambda$. By Lemma \ref{lem:pyramorder}, $\lambda \ntriangleleft \mu$ and $\lambda \ntriangleleft \m(\tau)$, which concludes the proof of this property. \\

We can now show that $\mathbf{D}_2$, $\mathbf{D}_3$ and $\mathbf{D}_5$ are matrices of zeros. Let $\lambda_i, \lambda_j \in V_\gamma^1=W_\gamma^1$ for some $1 \leq i \leq t$. From the affirmation above, $\lambda_i \ntriangleleft \m(\lambda_j)$. Then, by Proposition \ref{fact:unitri}, $d_{\lambda_i \m(\lambda_j)}=0$. That is, $\mathbf{D}_2= (0)_{t \times t}$. On the other hand, by Proposition \ref{fact:mulli} $d_{\lambda_i' \lambda_j}=d_{\lambda_i \m(\lambda_j)}=0$, so that $\mathbf{D}_3=\mathbf{D}_2= (0)_{t \times t}$.

Let $\mu_k \in W_\gamma^3$, for some $1\leq k \leq \frac{p-1}{2}$. From the assertion above, $\lambda_i \ntriangleleft \mu_k$. Then by Proposition \ref{fact:unitri}, $d_{\lambda_i \mu_k}=0$. On the other hand, $d_{\lambda_i' \mu_k}=d_{\lambda_i \m(\mu_k)}=d_{\lambda_i \mu_k}=0$. This shows that $\mathbf{D}_5= (0)_{2t \times \frac{p-1}{2}}$. \\

\noindent \underline{$\mathbf{D}_6$}:\ it remains to show that $\mathbf{D}_6$, is lower unitriangular. We first prove that it is lower triangular. Consider the matrix $\mathbf{\widetilde{D}}_6$ defined by taking $\mathbf{D}_6$ and ordering rows as $\nu_1, \nu_2, \ldots, \nu_\frac{p-1}{2}$ and columns as $\mu_1, \mu_2, \ldots, \mu_\frac{p-1}{2}$. The matrix $\mathbf{\widetilde{D}}_6$ is ``\emph{almost}'' lower triangular. For making clear what we mean by ``\emph{almost}'', let us study the precise form of the matrix, with Proposition \ref{prop:matrixrich} and Lemma \ref{lemma:drondb} (for clarity see Example \ref{ex:lastm} after this proof). From Proposition \ref{prop:matrixrich}, in $\mathbf{D}_\gamma$ every entry in the column $\mu_k$ is zero except for rows $\mu_k$, $\m(\mu_k)'=\mu_k'$ which are $1$, and any $\lambda$ with $\mu_k' \lessdom \lambda \lessdom \mu_k$ and $\dr\lambda - \dr\mu_k = \pm 1$, which is also $1$. Now if $\mu_k \in \dr_l$ then $\mu_k' \in \dr_l$, and for $\mathbf{\widetilde{D}}_6$ we are only interested in rows corresponding to the self-conjugate partitions in $\Bgot_\gamma$. Self-conjugate partitions and self-Mullineux partitions are never in a same set $\dr_l$, so that $\mu_k '$ is not self-conjugate. Hence we are only left with looking for partitions $\lambda$ with $\mu_k' \lessdom \lambda \lessdom \mu_k$ and $\dr\lambda - \dr\mu_k = \pm 1$, among self-conjugate partitions $\nu_1,\nu_2, \ldots, \nu_\frac{p-1}{2}$. Consider the possible two cases: $\delta=0$ of $\delta \geq 1$. 

If $\delta=0$, from Lemma \ref{lemma:drondb}, the partitions $\mu_1,\nu_1,\mu_2,\nu_2, \ldots \mu_\frac{p-1}{2},\nu_\frac{p-1}{2} $ belong respectively, in that same order, to sets $\dr_1, \dr_2, \dr_3, \dr_4, \ldots, \dr_\frac{p-3}{2}, \dr_\frac{p-1}{2}$. Then, the first column of $\mathbf{\widetilde{D}}_6$, column $\mu_1$, has possibly a $1$ only in row $\nu_1$, since $\dr\nu_1 - \dr\mu_k=1$, and the rest of entries in this column are equal to $0$ since $\dr\nu_i - \dr\mu_k=1$ for $i\neq 1$. For $1<k\leq \frac{p-1}{2}$, $\mu_k \in \dr_{2k-1}$ has only two $1$; one in row $\nu_{k-1}\in \dr_{2k-2}$ and one in row $\nu_{k}\in \dr_{2k}$. In this case, then $\mathbf{\widetilde{D}}_6$ takes the following form:
\[
\begin{array}{cccccc}
&\mu_1 & \mu_2 & \cdots & \mu_\frac{p-3}{2} &  \mu_\frac{p-1}{2} \\
\nu_1 & * & * &  &\\
\nu_2 & \cdot & * & \\
\vdots &  & & \ddots  \\
\nu_\frac{e-3}{2} &  & & & * & * \\
\nu_\frac{e-1}{2} &  & & & \cdot & *
\end{array}
\]

where ``$*$'' is either $0$ or $1$ (we will see that it is $1$) and dots are $0$. 

If $\delta \geq 1 $, the partitions $\nu_1,\mu_1,\nu_2,\mu_2, \ldots, \nu_\delta,\mu_\delta,\mu_{\delta+1},\nu_{\delta+1}, \ldots,  \mu_\frac{p-1}{2},\nu_\frac{p-1}{2}$ belong respectively, in that same order, to sets $\dr_1, \dr_2, \dr_3,$ $ \dr_4, \ldots, \dr_\frac{p-3}{2}, \dr_\frac{p-1}{2}$. Hence, for a similar reasoning, starting from column $\mu_{\delta+1}$, the matrix $\mathbf{\widetilde{D}}_6$ is of the same form as in the case $\delta=0$, but columns $\mu_k$ with $1\leq k \leq \delta$, have a $1$ in rows $\nu_k$ and $\nu_{k+1}$:
\[
\begin{array}{ccccccccc}
&\mu_1 & \mu_2 & \cdots & \mu_\delta & \mu_{\delta+1} & \mu_{\delta+2} & \cdots &  \mu_\frac{p-1}{2} \\
\nu_1 & * & \cdot\\
\nu_2 & * & *\\
\vdots &  & & \ddots\\
\nu_\delta & & & & * & \\ 
\nu_{\delta+1} & & & & & * & *\\
\nu_{\delta+2} & & & & & \cdot & *\\
\vdots & & & & &  & &\ddots\\
\nu_\frac{p-1}{2}  & & & & &  & && *
\end{array}
\]
In any of the two cases we can see that the total order 
\[
\nu_1 \succ \nu_2 \succ \cdots \succ \nu_{\delta-1} \succ \nu_\delta \succ \nu_\frac{p-1}{2} \succ  \nu_\frac{p-3}{2} \succ \cdots \succ \nu_{\delta+2} \succ \nu_{\delta+1},
\]
which is the order chosen in matrix $\mathbf{D}_6$, makes this matrix lower triangular.
It remains to prove that the entries in the diagonal are $1$. That is, for every $1 \leq k \leq \frac{p-1}{2}$, we have to prove that $d_{\nu_k\mu_k}=1$. For this, we use Tables $1$ and $2$ in \cite{richards}, adapted by Fayers in \cite[Proposition 3.1]{fayers}.
Recall that  $\nu_k=\Lr{\frac{p-1}{2}-k,\frac{p-1}{2}+k}$ and $\mu_k=\fn{\frac{p-1}{2}-k,\frac{p-1}{2}+k}$. From \cite[Proposition 3.1]{fayers}, we have in particular that if $\mu = \fn{i-1,j}$ and $\lambda= \lr{i-1,j}$ then $d_{\lambda\mu}=1$, since for $\mu=\mu_k$ and $\lambda=\nu_k$ we are exactly in that case, then $d_{\nu_k\mu_k}=1$. Hence $\mathbf{D}_6$ is lower unitriangular and this concludes the proof of Theorem \ref{th:main}.

\end{proof}

\begin{example}\label{ex:lastm}
Let $p=11$, $n=36$ and let $\gamma$ be the $11$-core $\gamma=(7,2,1^5)$. The self-Mullineux partitions in the block $\Bgot_\gamma$ of weight $2$ of $\F_{11}\mathfrak{S}_{36}$ are
\[
\mu_1=(9,8,6,4,3,2^3), \ \mu_2=(10,8,5,4,2^4,1), \  \mu_3=(12,8,4,2^5,1^2), \  \mu_4=(12,8,3,2^5,1^3), \ \mu_5=(18,2^7,1^4).
\]
They belong respectively to $\dr_2,\dr_4,\dr_6,\dr_7,\dr_9$.
The self-conjugate partitions in $\Bgot_\gamma$ are 
\[
\nu_1=(8^2, 6, 4, 3^2, 2^2), \ \nu_2=(9,8,5,4,3,2^3,1), \  \nu_3=(10,8,4^2,2^4,1^2), \  \nu_4=(12,8,2^6,1^4), \ \nu_5=(18,2,1^{16}).
\]
They belong respectively to $\dr_1,\dr_3,\dr_5,\dr_8,\dr_{10}$.
The matrix $\mathbf{\widetilde{D}}_\gamma$ is 
\[
\begin{array}{cccccc}
& \mu_1 & \mu_2 & \mu_3 & \mu_4 & \mu_5 \\
\nu_1 & 1 \\
\nu_2 & 1 & 1 \\
\nu_3 & \cdot & 1 & 1 \\
\nu_4 & \cdot & \cdot & \cdot & 1 & 1 \\
\nu_4 & \cdot & \cdot & \cdot & \cdot & 1 \\
\end{array}
\]
\demo
\end{example}

\bigskip
\noindent \textbf{Acknowledgements.} The author is grateful to Nicolas Jacon and Lo\"ic Poulain d'Andecy for the helpful discussions, ideas and careful reading and to Matthew Fayers for useful suggestions. The author is also thankful to the referee for their remarks.

\bibliographystyle{alpha}
\bibliography{biblio.bib}

\end{document}